\documentclass{elsarticle}
\usepackage{amssymb}

\usepackage{epic}
\usepackage{amsmath}
\usepackage{amsfonts}
\usepackage{pstricks}

\usepackage{float}
\usepackage{amsthm}
\usepackage{subfig}
\usepackage{graphicx}
\usepackage{multirow}
\usepackage{adjustbox}
\usepackage{color}
\usepackage[subfigure]{tocloft}
\usepackage{subfig}
\usepackage[utf8]{inputenc}
\usepackage{soul}
\usepackage{hyperref}
\hypersetup{colorlinks=true,linkcolor=blue,urlcolor=blue}

\newtheorem{theorem}{Theorem}

\def\bigO{\mathcal{O}}

\def\bigO{\mathcal{O}}

\title{Approximate Taylor methods for ODEs}

\journal{Computers \& Fluids}

\begin{document}

\begin{frontmatter}
\author[uv]{A.~Baeza}\ead{antonio.baeza@uv.es}
\author[uc] {S.~Boscarino }\ead{boscarino@dmi.unict.it}
\author[uv]{P.~Mulet}\ead{mulet@uv.es}
\author[uc]{G.~Russo}\ead{russo@dmi.unict.it}
\author[udec]{D.~Zor\'{\i}o}\ead{dzorio@ci2ma.udec.cl}
\date{}

\address[uv]{Departament de Matemàtiques,  Universitat de
	València. Av. Vicent Andr\'es Estell\'es, 46100 Burjassot  (Spain)}
\address[uc]{Department of Mathematics and Computer Science, University of Catania, 
  95125 Catania (Italy)}
\address[udec]{Centro de Investigación en Ingeniería Matemática (CI$^2$MA), Universidad de Concepción. Casilla 160-C, Concepción (Chile)}

\leavevmode\thispagestyle{empty}

\noindent The version of record of this article, first published in Computers \& Fluids, is available online at Publisher’s website: \url{https://doi.org/10.1016/j.compfluid.2017.10.001}

\newpage

\begin{abstract}
A new method for the numerical solution of ODEs is presented. This approach
is based on an approximate formulation of the Taylor methods that has a much
easier implementation than the original Taylor methods, since only the functions
in the ODEs, and not their derivatives, are needed, just as in classical Runge-Kutta schemes. 
Compared to Runge-Kutta methods, the number of function evaluations to achieve 
a given order is higher, however with the present procedure it is much easier to produce 
arbitrary high-order schemes, which may be important in some applications. 
In many cases the new approach leads
to an asymptotically lower computational cost when compared to the Taylor expansion based on exact derivatives.
The numerical results that are obtained with our proposal are
satisfactory and show that this approximate approach can attain results as
good as the exact Taylor procedure with less implementation and computational
effort.

\end{abstract}


\end{frontmatter}

\section{Introduction}
In this paper we develop a numerical scheme consisting in an
approximate formulation of  Taylor methods for ODEs, akin to the
scheme proposed by Zor\'{\i}o et. al. in  
\cite{ZorioBaezaMulet17} for hyperbolic conservation laws.

Taylor methods are based on differentiating
the ODEs to obtain the high-order derivatives of the solution to be
used in (truncated) Taylor series expansions for the next time step of
the solution. The terms in the differentiation  of a composition of
functions may grow exponentially as the number of terms  in the 
Taylor series does, thus leading to expressions that are
difficult to obtain and costly to evaluate.
The scheme proposed in this paper is much easier to implement due to
its main strong point, which is that its implementation only depends
on the functions in the ODEs and not on their high-order
derivatives. Moreover, this scheme can be formulated to 
an arbitrary high order of accuracy with relative easiness and the
computational cost increases quadratically as the order increases.
Compared to relatively low-order Runge-Kutta methods, the number of
function evaluations to achieve  a given order is higher, however with the present procedure it is much easier to produce 
arbitrary high-order schemes, a feature that may be crucial for
certain applications that require very 
high orders of accuracy and possibly extended precision arithmetic
(see, e.g., \cite{Jorba05,Paresetal17}).

This paper is organized as follows: in Section \ref{ens} we review
the Taylor methods for solving ODEs; in Section \ref{cka}  we present the approximate
Taylor methods and analyze their accuracy and stability; in Section \ref{nex}
some numerical experiments are presented in order to show the
performance of these schemes; finally, in Section \ref{cnc} some
conclusions are drawn.

\section{Taylor methods}\label{ens}

Without loss of generality, we consider initial value problems for
systems of $m$ autonomous  ODEs for
the unknown $u=u(t)$:
\begin{equation}\label{eq:edo}
    u'=f(u), u(0)=u_0,
\end{equation}
 for nonautonomous systems can be cast as autonomous systems by
considering $t$ as a new unknown and by correspondingly inserting a
new equation $t'=1$ and initial condition $t(0)=0$.

We aim to obtain approximations   $v_{h,n}\approx u(t_n) $ of the
solution $u$ of \eqref{eq:edo} on  $t_n=nh $, 
$n=0,\dots$,  with spacing $h>0$.
We use
the following notation for time derivatives of $u$ (where we omit for
simplicity the dependency on $h$):
   \begin{align*}
     u_{n}^{(l)}&=\frac{d^{l} u(t_n)}{d t^l}.
   \end{align*}

Taylor methods aim to obtain an $R$-th order accurate numerical scheme, i.e., a scheme
with a local truncation error 
of order $R+1$, using the Taylor expansion  of the solution $u$
from time $t_n$ to the next time $t_{n+1}$:
$$u_{n+1}^{(0)}=u(t_{n+1})=\sum_{l=0}^R\frac{h^l}{l!}u_{n}^{(l)}+\bigO(h^{R+1}).$$

Approximations  $v_{h,n}^{(l)}\approx u_{n}^{(l)}$, $0\leq l \leq R$, are substituted in 
the Taylor series,  with $v_{h,n}^{(0)}=v_{h,n}$, $v_{h,n}^{(1)} = f(v_{h,n}^{(0)})$ and
$v_{h,n}^{(l)}$,  $2\leq l\leq R$, computed by successive differentiation of $f(u)$ with respect to $t$. 
Truncation of the resulting series gives the $R-th$ order Taylor method
$$v_{h,n+1}=v_{h,n+1}^{(0)}=\sum_{l=0}^R\frac{h^l}{l!}v_{h,n}^{(l)}.$$

As an example, the Taylor method for \eqref{eq:edo} in the scalar case ($m=1$) for third order accuracy 
(see, for instance,
\cite{HairerNorsettWanner93} for more details) is computed as follows: 
the fact that $u$ solves \eqref{eq:edo} implies that
\begin{equation}\label{eq:2001}
 \begin{aligned}
   u''&=f(u)'=f'(u)u'\\
   u'''&=f(u)''=f''(u)(u')^2+f'(u)u'',
 \end{aligned}
\end{equation} 
which yields approximations
 \begin{equation}\label{eq:40}
  \begin{aligned}
u'_n\approx v_{h,n}^{(1)}&=f(v_{h,n}^{(0)})\\
u''_n\approx v_{h,n}^{(2)}&=f'(v_{h,n}^{(0)})v_{h,n}^{(1)}\\
u'''_n\approx v_{h,n}^{(3)}&=f''(v_{h,n}^{(0)})(v_{h,n}^{(1)})^2+f'(v_{h,n}^{(0)})v_{h,n}^{(2)}.
\end{aligned}
\end{equation}
that result in the third order Taylor method
\begin{equation}\label{eq:taylor3e}
\begin{aligned}
v_{h,n+1}&=v_{h,n}+hv_{h,n}^{(1)}+\frac{h^2}{2}v_{h,n}^{(2)}+\frac{h^3}{6}v_{h,n}^{(3)}.
\end{aligned}
\end{equation}

 The following result
\cite{faadibruno1857}, which is a
generalization of the chain rule, describes a procedure to compute
high-order derivatives of compositions of functions. 

\begin{theorem}[Faà di Bruno's formula]\label{th:faadibruno}
  Let $f:\mathbb{R}^m\rightarrow\mathbb{R},$
  $u:\mathbb{R}\rightarrow\mathbb{R}^m$ $r$ times continuously
  differentiable. Then
\begin{equation}\label{eq:faadibruno}
\frac{d^r f(u(t))}{dt^r}=\sum_{s\in\mathcal{P}_r}\left(\begin{array}{c}
      r \\
      s
      \end{array}\right)f^{(|s|)}(u(t))D^su(t),
  \end{equation}    
  where $\mathcal{P}_{r}=\{ s\in\mathbb N^{r} /
\sum_{j=1}^{r} j s_j=r \}$, $|s|=\sum_{j=1}^{r} s_j$,
$\displaystyle\left(\begin{array}{c}
      r \\
      s
      \end{array}\right)=\frac{r!}{s_1!\cdots s_r!}$, 
    $D^s u(t)$ is an $m\times |s|$ matrix whose
    ($\sum\limits_{l<j}s_l+i$)-th column is given by
\begin{equation}\label{eq:ds}
  \displaystyle (D^s u(t))_{\sum\limits_{l<j}s_l+i}=\frac{1}{j!}\frac{\partial^{j}
      u(x)}{\partial t^j},\\
    \quad i=1,\dots,s_j,\quad j=1,\dots,r,
  \end{equation}
  and the action of the $k$-th derivative tensor of $f$ on a
    $m\times k$ matrix $A$ is  given by 
  \begin{equation}\label{eq:77}
      f^{(k)}(u)A=\sum_{i_1,\dots,i_k=1}^{m}\frac{\partial^k f}{\partial
    u_{i_1}\dots\partial u_{i_k}}(u) A_{i_1,1}\dots A_{i_{k},k}.
\end{equation}
\end{theorem}

\subsection{Computational cost}
\label{ss:computational_cost}
Fa\`a di Bruno's formula can be used to obtain efficient
implementations of Taylor methods if the derivatives $f^{(k)}(u)$ are
available. These are typically obtained by  symbolic calculations.

Specifically,   the implementation
of a Taylor method of order $R$ requires  $f_i^{(r)}(u)$, $i=1,\dots,m$, for
$r=1,\dots,R-1$. It is not easy to give a simple expression for the
number of terms in \eqref{eq:faadibruno}, but since $s_r=(r, 0,\dots,0)\in
\mathcal{P}_r$,  the terms
\begin{equation}\label{eq:800}
f_{i}^{(r)}(v_{h,n})\overbrace{\begin{bmatrix}v_{h,n}^{(1)} &\dots &
  v_{h,n}^{(1)}
\end{bmatrix}}^{r}, \quad i=1,\dots,m
\end{equation}
corresponding to
\begin{equation*}
f_{i}^{(|s_r|)}(u(t))D^{s_r}u(t) = f_{i}^{(r)}(u(t))[\overbrace{u'(t) \dots
  u'(t)}^{r} ]
\end{equation*}
have to be computed.
 Assuming that no 
derivative of $f_{i}$ vanishes and taking into account that
the number of distinct $r$-th order derivatives is $\binom{m+r-1}{r}$,
an efficient computation of \eqref{eq:77}, \eqref{eq:800} 
would require about  $(r-1)m^r$ products to account for
$A_{i_1,1}\dots A_{i_{r},r}$, $m^r$ sums and $\binom{m+r-1}{r}$
computations of $r$-th order derivatives, $r\geq 0$, and corresponding products.
Assuming a uniform  computational cost for these $r$-th order
derivatives of $c_r$ operations, a fairly conservative lower bound
for the
computational cost for obtaining $v_{h,n}^{(r)}$, $r=2,\dots, R$ is
\begin{equation*}
  \begin{aligned}
    &m\sum_{r=1}^{R-1}\left((r-1)m^r+m^r+
  \binom{m+r-1}{r}(c_r+1)\right)\\  
    &=m\sum_{r=1}^{R-1}\left(rm^r+
  \binom{m+r-1}{r}(c_r+1)\right).
\end{aligned}
\end{equation*}
The computation of the first order term $f_i$ is $c_0$ and the
computation of the Taylor polynomial of order $R$ for each component
is $2R$, so the overall cost can be bounded from below by
\begin{equation}\label{eq:801}
  \begin{aligned}
    C_{T}(R, m)&=m\sum_{r=0}^{R-1}\left(rm^r+
      \binom{m+r-1}{r}(c_r+1)\right).
  \end{aligned}
\end{equation}

For $m=1$, 
\begin{equation}\label{eq:802}
  \begin{aligned}
    C_{T}(R, 1)&\approx R^2 +\sum_{r=0}^{R-1}c_r.
  \end{aligned}
\end{equation}

For $m>1$
\begin{equation}\label{eq:803}
  \begin{aligned}
    C_{T}(R, m)&\geq \widetilde{C}_{T}(R, m)=(R-1)m^R+m\sum_{r=0}^{R-1}
  \binom{m+r-1}{r}(c_r+1).
\end{aligned}
\end{equation}

\section{Approximate Taylor methods}\label{cka}
We now present an alternative to the Taylor methods described in Section \ref{ens}, which is much less expensive for large
$m, R$
and agnostic about the equation, in the sense that its only
requirement is the knowledge of the function $f$ defining the ODE.
As proven below in Theorem \ref{th:1}, this technique is based on the
fact that the replacement of  exact expressions 
of high-order derivatives of the function $f$
  that appear in  Taylor methods, e.g., those in \eqref{eq:2001} for
  the third-order Taylor method, 
 by 
 approximations  
$v_{h,n}^{(l)}\approx u_{n}^{(l)}$, $l=1,\dots,R$, such that 
\begin{align}\label{eq:200}
  v_{h,n}^{(1)} = f(v_{h,n}^{(0)}),\quad 
v_{h,n}^{(l)}=u_{n}^{(l)}+\bigO(h^{R+1-l}), l=2,\dots,R,
\end{align}
when $v_{h,n}^{(0)}=u(t_n)$, yields an $R-th$ order accurate scheme. In this work 
we compute the required approximations 
 by simply using finite differences of sufficient order.
Notice that the \textit{exact} third order Taylor method, defined
through \eqref{eq:40} and \eqref{eq:taylor3e}, 
satisfies \eqref{eq:200} with no error.

We introduce next some notation which will help in the description of
the approximate method along this section. 

For a function $u\colon \mathbb R\to \mathbb R^{m}$, we denote the 
function on  the grid defined by a base
  point $a$ and grid space $h$ by
  \begin{equation*}
    G_{a,h}(u)\colon\mathbb{Z} \to \mathbb {R}^{m},\quad
    G_{a,h}(u)_i=u(a+ih).
  \end{equation*}
For naturals $p,q$, we denote by  $ \Delta^{p,q}_{h}$
the centered finite differences operator that  approximates $p$-th order
  derivatives to order $2q$ on grids with spacing $h$, which, for any $u$
  sufficiently differentiable,  satisfies (see \cite[Proposition
  1]{ZorioBaezaMulet17} for the  details):
  \begin{align}\label{eq:3}
    \Delta^{p,q}_{h}
    G_{a,h}(u)=u^{(p)}(a)+\bigO(h^{2q}).
  \end{align}

 We aim to define approximations $v^{(k)}_{h,n} \approx
u^{(k)}_{n}$, $k=0,\dots,R$, recursively. Given $v_{h,n}$, we start the recursion with
\begin{equation}\label{eq:151}
  \begin{aligned}
    v^{(0)}_{h,n}&=v_{h,n},\\
    v^{(1)}_{h,n}&=f(v_{h,n}).\\
  \end{aligned}
\end{equation}   
Associated to fixed $h, n$, once obtained $v^{(l)}_{h,n}$,
$l=0,\dots,k$ in the recursive process we define
\begin{equation}\label{eq:15}
  \begin{aligned}
    v^{(k+1)}_{h,n} &=
    \Delta_{h}^{k,\left\lceil \frac{R-k}{2}\right\rceil}\Big(G_{0,h}\big(f(T_{h,n}^{k})\big)\Big),\\
  \end{aligned}
\end{equation}   
where $T_{h,n}^{k}$ is the $k$-th degree approximate Taylor polynomial given by
\begin{align}\label{eq:300}
  T_{h,n}^{k}(\rho)=\sum_{l=0}^{k}\frac{v^{(l)}_{h,n} }{l!} \rho^l.
\end{align}

With this notation, the proposed scheme is:
\begin{equation}\label{eq:60}
v_{h,n+1}=v_{h,n}+\sum_{l=1}^R\frac{h^l}{l!}v_{h,n}^{(l)}.
\end{equation}

The following result is a simplified adaptation of the corresponding
result in \cite{ZorioBaezaMulet17}.

\begin{theorem}\label{th:1}
  The scheme defined by \eqref{eq:15} and \eqref{eq:60} is $R$-th
  order accurate. 
\end{theorem}

\begin{proof}
  For the accuracy analysis of the local truncation error, we take
  \begin{equation}\label{eq:45}
v^{(0)}_{h,n}=u(t_n),
\end{equation}
and use induction on $k=1,\dots,R$ to prove that
  \begin{align}
    \label{eq:47}
v^{(k)}_{h,n}&= u^{(k)}_{n}+\bigO(h^{R-k+1}).
\end{align}
 The result  in
\eqref{eq:47} for $k=1$ immediately follows, with no error, from the
fact that $u$ solves \eqref{eq:edo}.

Assume now the result to hold for $k$ and aim to prove it for
$k+1\leq R$.
From
\eqref{eq:15} and \eqref{eq:3}, with
$q=\left\lceil   \frac{R-k}{2}\right\rceil$
and the notation $w=T_{h,n}^{k}$: 
  \begin{align}\label{eq:1}
    v^{(k+1)}_{h,n} &=
    f(w)^{(k)}(0)+\bigO(h^{2q}).
  \end{align}
Fa\`a di Bruno's formula \eqref{eq:faadibruno} applied to $f(w)$ and
$f(u)$ yields: 
\begin{equation}     \label{eq:5}
   \begin{aligned}
    &f(w)^{(k)}(0)=\sum_{s\in \mathcal{P}_{k}}
    \left(
      \begin{array}{c}
        k\\
        s
      \end{array}
    \right)
    f^{(|s|)}(w(0))       \big(D^{s} w(0)\big),\\
    &D^{s} w (0)=
    \begin{bmatrix}
      \overbrace{
        \begin{array}{ccc}
          \frac{w^{(1)}(0)}{1!}
          &\dots&
          \frac{w^{(1)}(0)}{1!}
        \end{array}
      }^{s_1}
      &\dots&
      \overbrace{
        \begin{array}{ccc}
          \frac{w^{(k)}(0)}{k!}
          &\dots&
          \frac{w^{(k)}(0)}{k!}
        \end{array}
      }^{s_k}
    \end{bmatrix},      \\
   &D^{s} w (0)= 
    \begin{bmatrix}
      \overbrace{
        \begin{array}{ccc}
          \frac{w_{h,n}^{(1)}}{1!}
          &\dots&
          \frac{v_{h,n}^{(1)}}{1!}
        \end{array}
      }^{s_1}
      &\dots&
      \overbrace{
        \begin{array}{ccc}
          \frac{v_{h,n}^{(k)}}{k!}
          &\dots&
          \frac{v_{h,n}^{(k)}}{k!}
        \end{array}
      }^{s_k}
    \end{bmatrix},     
  \end{aligned}       
\end{equation}
  
  \begin{equation}\label{eq:4}
  \begin{aligned}
    &f(u)^{(k)}(t_n)=
    \sum_{s\in \mathcal{P}_{k}}
    \left(
      \begin{array}{c}
        k\\
        s
      \end{array}
    \right)
    f^{(|s|)}(u(t_n))       \big(D^{s} u(t_n)\big),\\
    &D^{s} u (t_n)=
    \begin{bmatrix}
      \overbrace{
        \begin{array}{ccc}
          \frac{u^{(1)}(t_n)}{1!}
          &\dots&
          \frac{u^{(1)}(t_n)}{1!}
        \end{array}
      }^{s_1}
      &\dots&
      \overbrace{
        \begin{array}{ccc}
          \frac{u^{(k)}(t_n)}{k!}
          &\dots&
          \frac{u^{(k)}(t_n)}{k!}
        \end{array}
      }^{s_k}
    \end{bmatrix},      \\
    &D^{s} u (t_n)=
    \begin{bmatrix}
      \overbrace{
        \begin{array}{ccc}
          \frac{u_{n}^{(1)}}{1!}
          &\dots&
          \frac{u_{n}^{(1)}}{1!}
        \end{array}
      }^{s_1}
      &\dots&
      \overbrace{
        \begin{array}{ccc}
          \frac{u_{n}^{(k)}}{k!}
          &\dots&
          \frac{u_{n}^{(k)}}{k!}
        \end{array}
      }^{s_k}
    \end{bmatrix}.     
  \end{aligned}
\end{equation}
We have $ w(0)=v_{h,n}^{(0)}=u(t_n)$ and, by induction, i.e.,
\eqref{eq:47}, we have:
  \begin{equation}\label{eq:90}
    v_{h,n}^{(l)}=u_{n}^{(l)}+\bigO(h^{R-l+1}),\quad l=1,\dots,k.
\end{equation}
For any $s\in\mathcal{P}_{k}$, $D^{s}w(0)$ is a
  $m\times |s|$ matrix, and  for any
  $\mu\in\{1,\dots,m\}$ and 
  $\nu\in\{1,\dots,|s|\}$, we have from \eqref{eq:ds}, \eqref{eq:5},
  \eqref{eq:4} and \eqref{eq:90} that 
  \begin{equation}\label{eq:78}
(D^s  w(0)-D^s u(t_n))_{\mu, \nu} = \frac{(v_{h,n}^{(l)}-u_{n}^{(l)})_{\mu}}{l!} =  
\bigO(h^{R-l+1}), \quad l=l(s, \nu)\leq k.
\end{equation}

We deduce from \eqref{eq:78}, \eqref{eq:77},
\eqref{eq:5}, \eqref{eq:4} that
  \begin{align}    \label{eq:82}
    f(w)^{(k)}(0)-f(u)^{(k)}(t_n)&=
    \sum_{s\in \mathcal{P}_{k}}
    \left(
      \begin{array}{c}
        k\\
        s
      \end{array}
    \right)
    f^{(|s|)}(u(t_n))       \big(D^{s} w(0)-D^{s} u(
    t_n)\big)= \bigO(h^{R-k+1}).
  \end{align}    
  
  Now, since $2q\leq R-k$, \eqref{eq:1}, \eqref{eq:4} and \eqref{eq:82} yield:
  \begin{equation*}
    v^{(k+1)}_{h,n}
 -u^{(k+1)}_{n}=\bigO(h^{R-k+1})+\bigO(h^{2q}) = \bigO(h^{R-k}),
\end{equation*}
which concludes the proof of \eqref{eq:47}.

The local truncation error is given by
  \begin{align*}
    E_L(t_n, h)&=u_{n+1}^{(0)}-    
    \sum_{l=0}^R\frac{ h^l}{l!}v_{h,n}^{(l)},
  \end{align*}
  where $v_{h,n}^{(l)}$ are computed from
  $v_{h,n}^{(0)}=u(t_n)$. Taylor expansion of the first
  term and the estimates in \eqref{eq:47} yield that  
  \begin{align*}
    E_L(t_n, h)&=\sum_{l=1}^{R}\frac{h^l }{l!}(u_{n}^{(l)}
    -v_{h,n}^{(l)})+\bigO(h^{R+1})\\
    &=\sum_{l=1}^{R}\frac{h^l }{l!}\bigO(h^{R-l+1})+\bigO(h^{R+1}) =\bigO(h^{R+1}).
  \end{align*}
\end{proof}

\subsection{Computational cost}\label{sec:cc}
  From \eqref{eq:15} and the centered finite difference formulas
  obtained in \cite[Proposition 1]{ZorioBaezaMulet17} we may deduce
  that  the
  computation   of  $v^{(l)}_{h,n}$, for
  $l=1,\dots,k$ requires about $R^2$  evaluations of $f$,
  and  $4mR^2$ floating point
  operations. The   $2mR$  floating point
  operations for the evaluation of the polynomial in
  \eqref{eq:60} are neglected in this computation, for having a lower
  order in $R$ that the previous terms. With the notation of section
  \ref{ss:computational_cost}, the approximate computational cost is
  \begin{equation}\label{eq:804}
    C_{AT}=mR^2 (c_0 + 4).
  \end{equation}
  A comparison of this expression with \eqref{eq:802} and
  \eqref{eq:803} leads to the   conclusion that our proposal should have a
  much smaller computational cost for large $m$ for tightly coupled
  systems, for which very few or none of the derivatives of the ODE
  vanish. On the contrary, for $m=1$, the advantage is not clear and
  it will depend on the ratios $c_r / c_0$, $r=1,\dots,R-1$.
  
  We consider, for instance, the application of this for the system of
  rational equations
  \begin{equation}\label{eq:806}
  \begin{aligned}
    u'_{i}&=f_i(u_1,\dots,u_{m})=\frac{p_i(u_1,\dots,u_{m})}{q_i(u_1,\dots,u_{m})},\\
    p_i(u_1,\dots,u_{m})&=\sum_{j}\alpha_{i,j}u_j, \quad
    q_i(u_1,\dots,u_{m})=\sum_{j}\beta_{i,j}u_j.
  \end{aligned}
\end{equation}  
It can be seen by using the general Leibniz rule that computing
  an $r$-th order partial derivative, for $r\geq 2$, from already
  computed lower order 
  derivatives needs $2r$ operations ($r-1$ sums, $r$ products by the
  coefficients $\beta$ and one extra product by the inverse of the
  denominators), so that one can take $c_r=2r$ in \eqref{eq:803}, for
  $r\geq 2$. First-order derivatives can be computed with $c_1=3$ operations
  and function values $f_i(u_1,\dots,u_m)$ take about $4m$ operations,
  so that $c_0=4m$. With this numbers, the quotient of \eqref{eq:803}
  and \eqref{eq:804} is
  \begin{multline}\label{eq:805}
Q(m, R)=\frac{\widetilde{C}_{T}(R, m)}{C_{AT}(R, m)}=\frac{(R-1)m^R+m\sum_{r=0}^{R-1}
  \binom{m+r-1}{r}(c_r+1)}{  mR^2 (c_0 + 4)  } \\
\geq 
\frac{(R-1)m^{R-1}+4m+4(m+1)+\sum_{r=2}^{R-1}
  \binom{m+r-1}{r}(2r+1)\big)}{  4(m+1)R^2  }\\
\geq \frac{(R-1)m^{R-1}+8m+\sum_{r=2}^{R-1}
  \binom{m+r-1}{r}(2r+1)}{  4(m+1)R^2  }.
\end{multline}
Since, for $R>3$, $\lim_{m\to\infty} Q(m, R)=\infty$ and, for $m>1$,
$\lim_{R\to\infty} Q(m, R)=\infty$, the computational savings for our
proposal are asymptotically guaranteed in a case like the present one,
in which the unknowns are strongly coupled.

\subsection{Low order approximate Taylor methods}
For notational simplicity, we drop the subindex $h,n$ from $v_{h,n}^{(l)}$.
Following the procedure  in \eqref{eq:15}--\eqref{eq:60}, we have
$$v^{(1)}=f(v).$$

We compute approximations of the second order time
derivative of $u$, $u^{(2)}$, by performing the following
operation:
\begin{equation}\label{eq:901}
  \begin{split}
    v^{(2)}=\frac{f(v+hv^{(1)})-f(v-hv^{(1)})}{2h}.
  \end{split}
\end{equation}
Therefore, the second order approximate Taylor method reads as:
\begin{align*}
v_{n+1}&=v+hv^{(1)}+\frac{h^2}{2}v^{(2)}\\
&=v+h\Big (f(v)+\frac{1}{4}f(v+hv^{(1)})-\frac{1}{4}f(v-hv^{(1)})\Big).
\end{align*}
Notice that this method can be regarded as a Runge-Kutta method with 3
stages and the following Butcher array:
$$
\begin{array}{r|rrr} 0 & 0 & 0 & 0\\
 -1 & -1 & 0 & 0\\
 1 & 1 & 0 & 0\\\hline
  & 1 & - \frac{1}{4} & \frac{1}{4} \end{array}.
$$
It is also worth pointing out that the application of this method to
the equation $u'=\lambda u$ results in the time-stepping
\begin{equation*}
  v_{n+1}=(1+h\lambda+\frac{1}{2}(h\lambda)^2)v_n,
\end{equation*}
and therefore the second order approximate Taylor method has the same absolute stability
region as its exact counterpart.

The approximation of $u^{(3)}$ is obtained in a similar
fashion:
\begin{equation}\label{eq:902}
  \begin{split}
    v^{(3)}&=\frac{f(v+hv^{(1)}+\frac{h^2}{2}v^{(2)})-2f(v)+f(v-hv^{(1)}+\frac{h^2}{2}v^{(2)})}{h^2}.
  \end{split}
\end{equation}

The next time step is then computed through the third order Taylor
expansion replacing the derivatives with their corresponding
approximations in \eqref{eq:901} and \eqref{eq:902}:
\begin{align*}
v_{n+1}&=v+hv^{(1)}+\frac{h^2}{2}v^{(2)}+\frac{h^3}{6}v^{(3)}\\
&=v+h\Big (f(v)+\frac{1}{4}\big(f(v+hv^{(1)})-f(v-hv^{(1)})\big)\\
&+\frac{1}{6}\big(f(v+hv^{(1)}+\frac{h^2}{2}v^{(2)})-2f(v)+f(v-hv^{(1)}+\frac{h^2}{2}v^{(2)})\big)\Big)\\
&=v+h\Big (\frac{2}{3}f(v)+\frac{1}{4} f(v+hv^{(1)})-\frac{1}{4}f(v-hv^{(1)})\\
&+\frac{1}{6}f(v+hv^{(1)}+\frac{h^2}{2}v^{(2)})+\frac{1}{6}f(v-hv^{(1)}+\frac{h^2}{2}v^{(2)})\Big).
\end{align*}
It can be readily seen that the  third-order approximate Taylor
method is a 5-stages Runge-Kutta method with Butcher array given by:
$$\begin{array}{r|rrrrr} 0 & 0 & 0 & 0 & 0 & 0\\
 -1 & -1 & 0 & 0 & 0 & 0\\
 1 & 1 & 0 & 0 & 0 & 0\\
 -1 & -1 & - \frac{1}{4} & \frac{1}{4} & 0 & 0\\
 1 & 1 & - \frac{1}{4} & \frac{1}{4} & 0 & 0\\ \hline
  & \frac{2}{3} & - \frac{1}{4} & \frac{1}{4} & \frac{1}{6} & \frac{1}{6} \end{array}
$$
We conjecture that the $R$-th order approximate Taylor method  can be
cast as a Runge-Kutta method with $R^2+\bigO(R)$ stages, which  is
larger than the upper bound on the number of stages to achieve $R$-th
order given in \cite[Theorem 324C,  page 188]{Butcher2008}: $\frac{3}{8}R^2+\bigO(R)$.

\section{Numerical experiments}\label{nex}

We analyze the exact and approximate Taylor methods with several numerical examples. We compare 
the order of the numerical error and the performance of both
algorithms. For these experiments, in  
the cases where the exact solution is unknown, we take  as reference solution the one computed 
with the same method with a mesh refined by a factor of 10 with respect to the finer 
discretization considered in the experiment. Regarding performance, the times indicated 
correspond to the average of multiple runs. The exact Taylor method was computed with the help of
MATLAB's symbolic toolbox. The \texttt{simplify} command was used on
each term of the Taylor expansion.

\subsection{Scalar equations}

We consider in this subsection three problems with different complexity to analyze the method. As a 
first example we consider the problem 

\begin{equation}\label{eq:ex1}
\left\{\begin{array}{l}
u'=\sin(u),\\
u(t_0) = u_0.
\end{array}\right.
\end{equation}
whose exact solution is given by $$u(t) = \cot ^{-1}\left(e^{t0-t} \cot \left(\frac{u_0}{2}\right)\right).$$
The results for $t_0=0, u_0 = \frac{\pi}{2}$, $T=1$ can be seen in Figure \ref{fig:sin1}. 

\begin{figure}
  \begin{center}
    \begin{tabular}{cc}
      \includegraphics[width=0.5\textwidth]{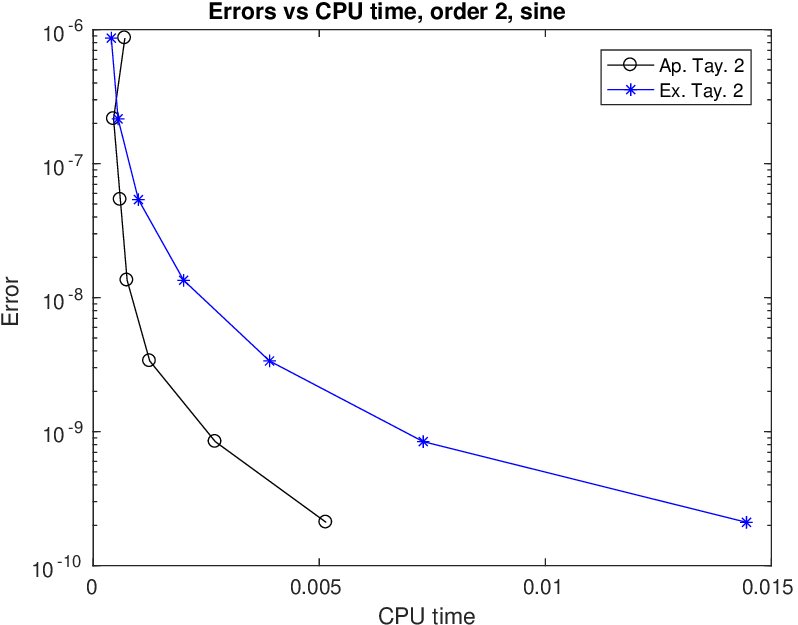}
      &\includegraphics[width=0.5\textwidth]{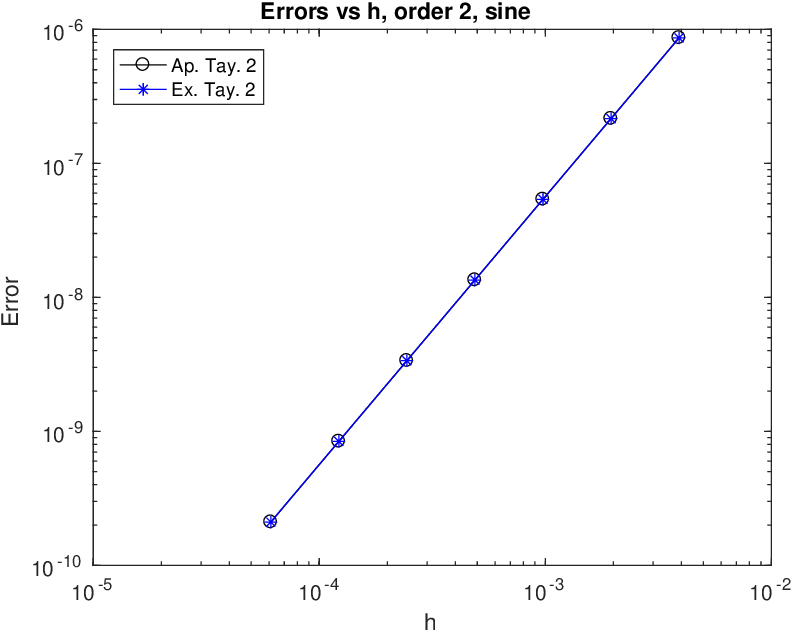}\\
      (a)&(b)\\
      \includegraphics[width=0.5\textwidth]{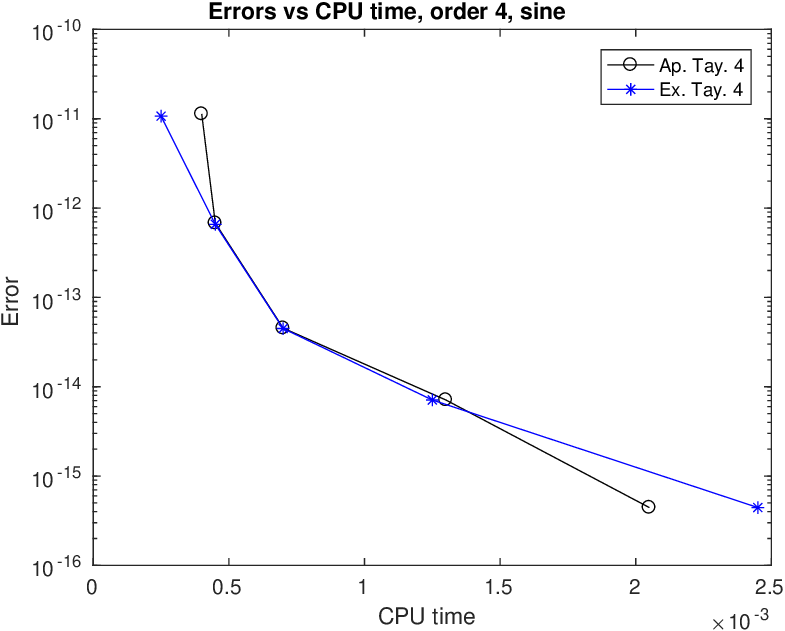}&
      \includegraphics[width=0.5\textwidth]{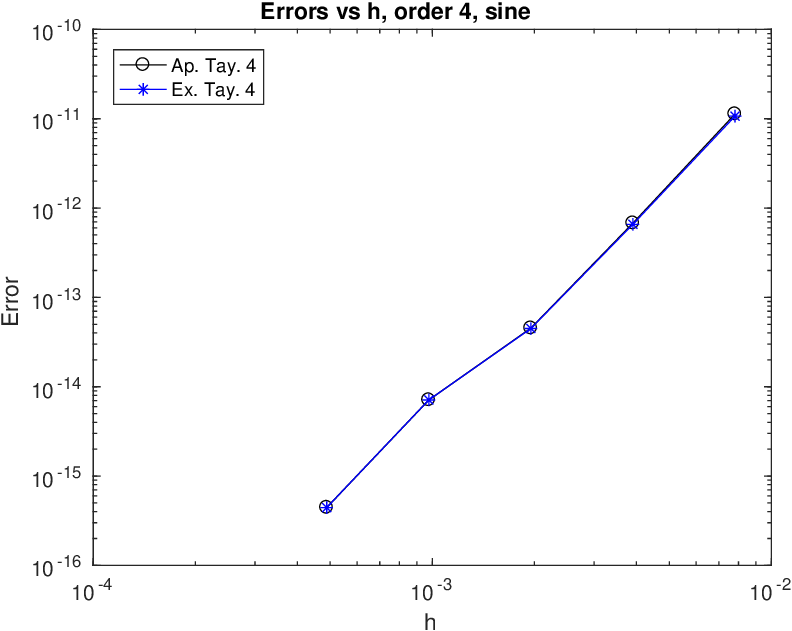}\\
      (c)&(d)\\
      \includegraphics[width=0.5\textwidth]{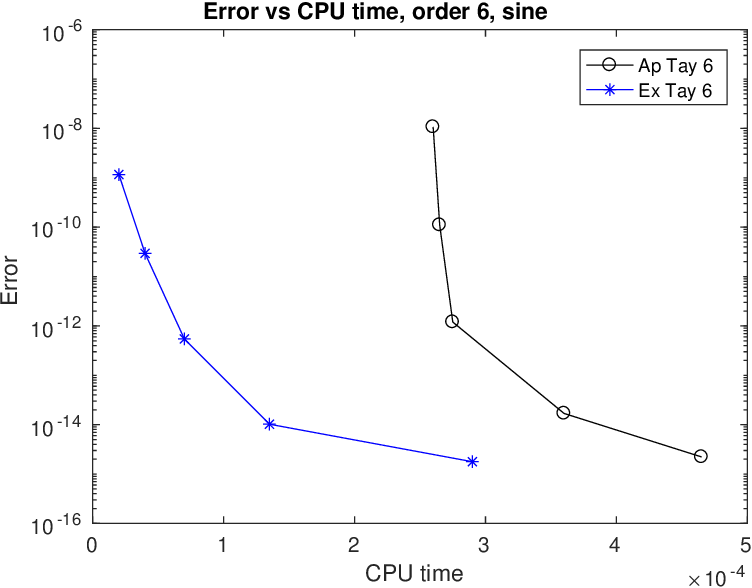}&
      \includegraphics[width=0.5\textwidth]{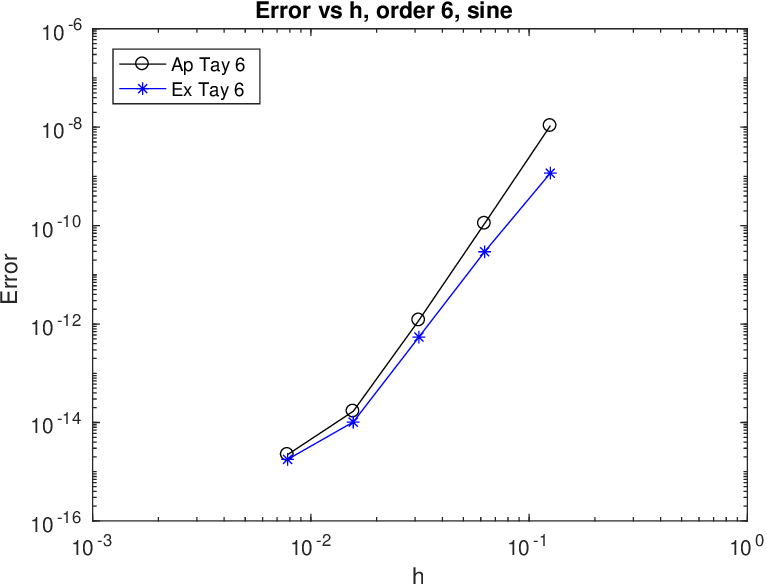}\\
      (e)&(f)	
    \end{tabular}
    \caption{Problem \eqref{eq:ex1}. Left: Error vs CPU time, Right: Error vs. $h$. (a)-(b) $R=2$; (c)-(d) $R=4$; (e)-(f) $R=6$.}
    \label{fig:sin1}
  \end{center}
\end{figure}

The second example is given by a Ricatti equation:
\begin{equation}\label{eq:ex2}
\left\{\begin{array}{l}
u'=- 2\,t\,u + u^2 +t^2 + 1,\\
u(t_0) = u_0,
\end{array}\right.
\end{equation}
with exact solution 
$$u(t) = \frac{1}{C-t}+t,\, C=\frac{1+(u_0-t_0) t_0}{u_0-t_0},\, u_0<t_0.$$

The results for this problem corresponding to $t_0=2, u_0 = 1$, $T=10$ are shown in Figure \ref{fig:ric1}.

\begin{figure}
  \begin{center}
    \begin{tabular}{cc}
      \includegraphics[width=0.5\textwidth]{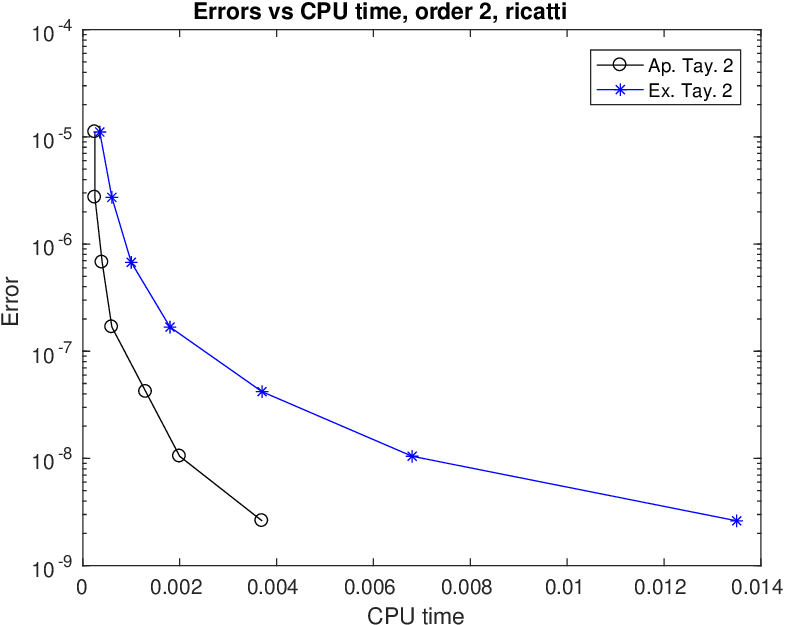}
      &\includegraphics[width=0.5\textwidth]{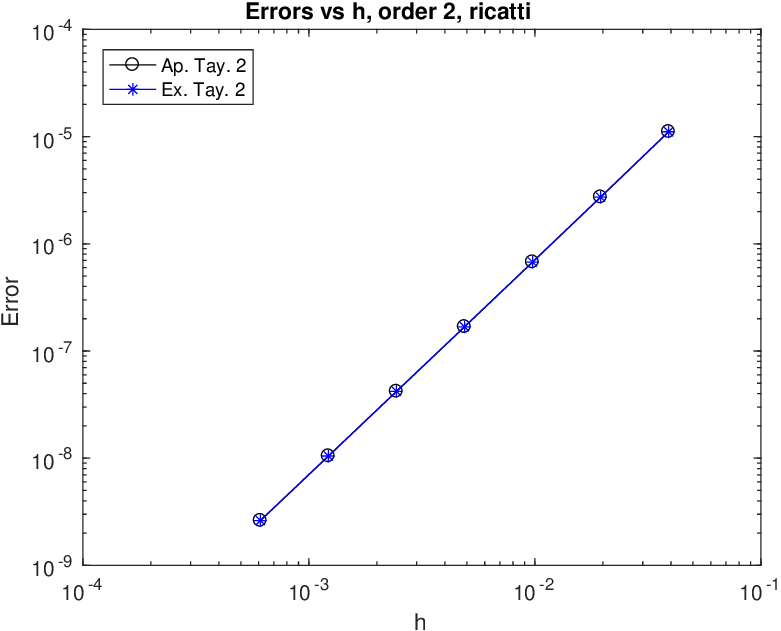}\\
      (a)&(b)\\
      \includegraphics[width=0.5\textwidth]{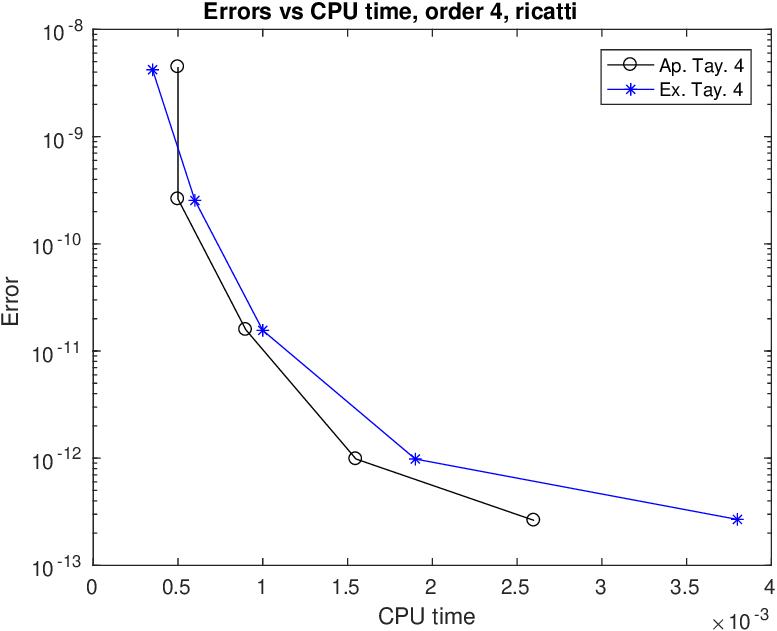}&
      \includegraphics[width=0.5\textwidth]{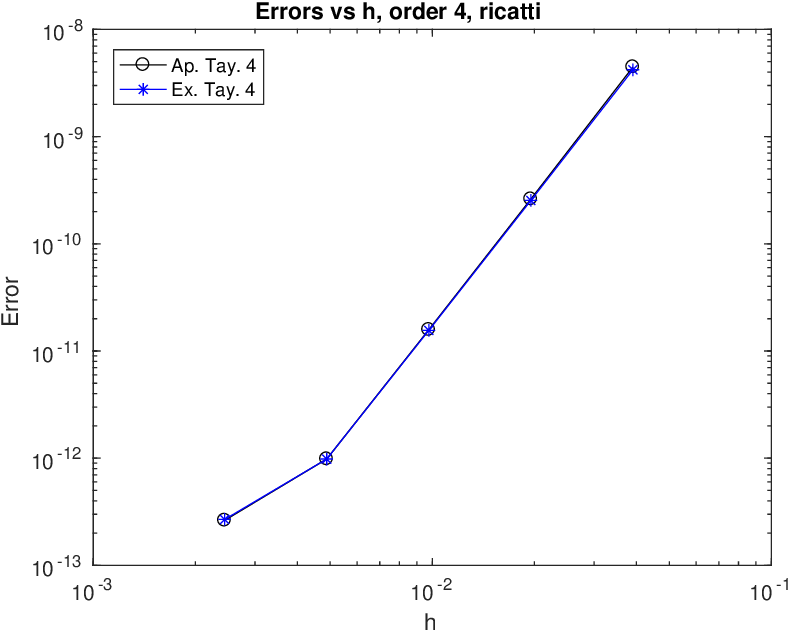}\\
      (c)&(d)\\
      \includegraphics[width=0.5\textwidth]{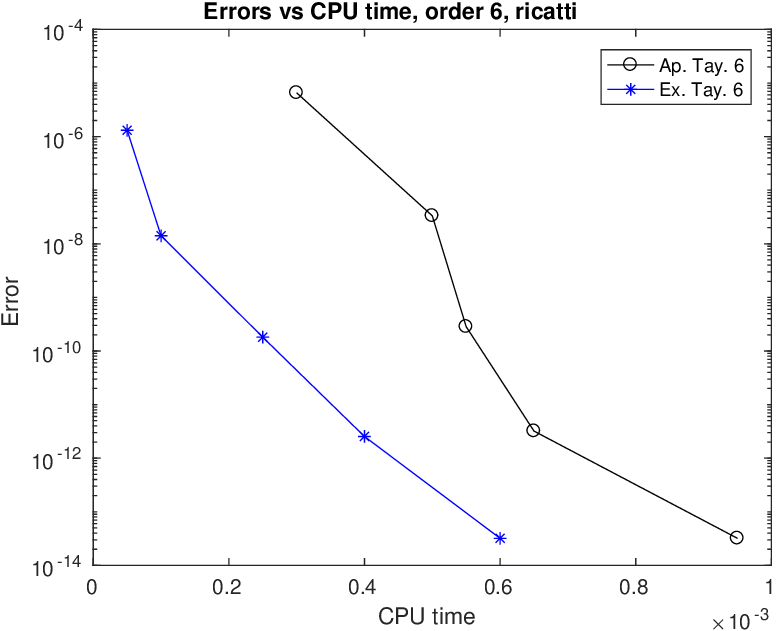}&
      \includegraphics[width=0.5\textwidth]{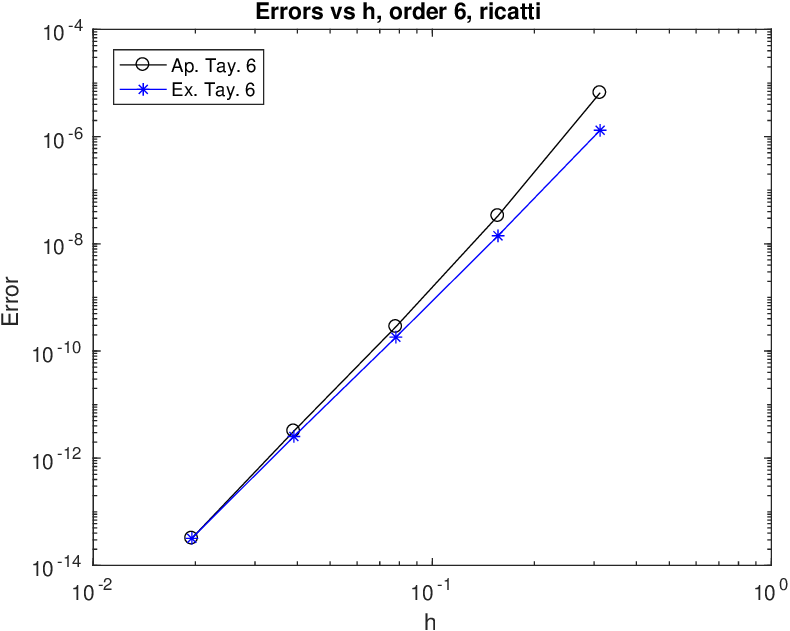}\\
      (e)&(f)	
    \end{tabular}
    \caption{Problem \eqref{eq:ex2}. Left: Error vs CPU time, Right: Error vs. $h$. (a)-(b) $R=2$; (c)-(d) $R=4$; (e)-(f) $R=6$.}
    \label{fig:ric1}
  \end{center}
\end{figure}

The last scalar example corresponds to an equation involving a logarithm:

\begin{equation}\label{eq:ex3}
\left\{\begin{array}{l}
u'=\frac{u}{t} \log\left(\frac{u}{t}\right),\\
u(t_0) = u_0.
\end{array}\right.
\end{equation}

The exact solution of \eqref{eq:ex3} is given by $$u(t) = t e^{Ct+1}, \quad C=\frac{\log\left(\frac{y0}{t0}\right)-1}{t0}, \quad t_0>0, u_0>0.$$

 Figure \ref{fig:log1} shows the results obtained for the case $t_0=1, u_0 = 1, T=8$.

\begin{figure}
  \begin{center}
    \begin{tabular}{cc}
      \includegraphics[width=0.5\textwidth]{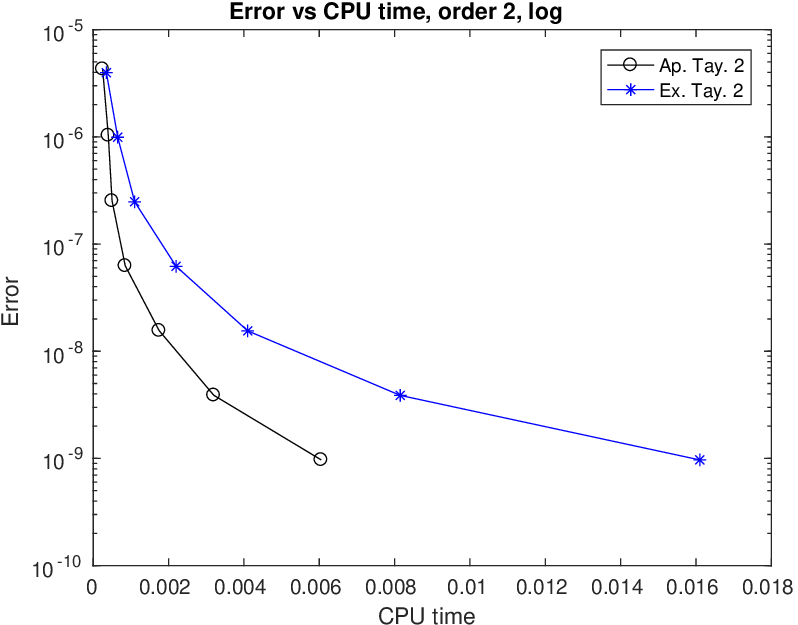}
      &\includegraphics[width=0.5\textwidth]{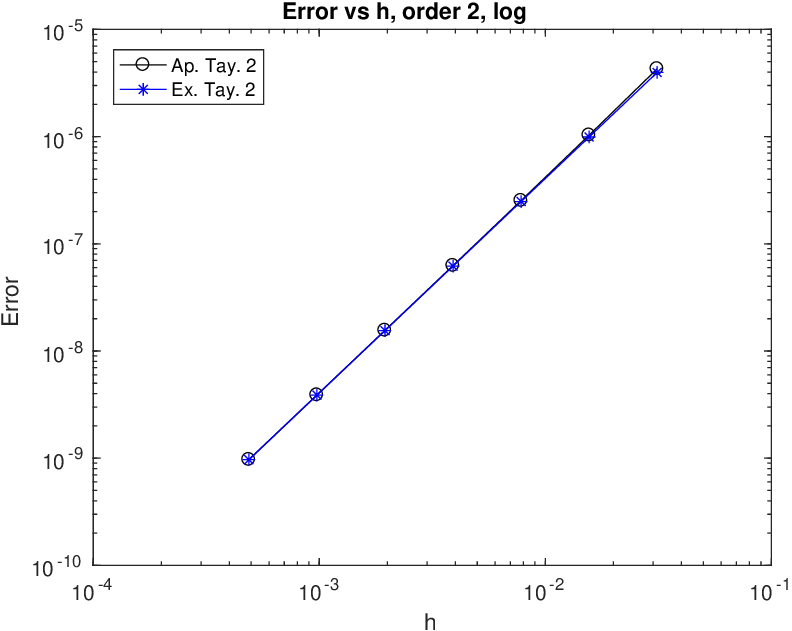}\\
      (a)&(b)\\
      \includegraphics[width=0.5\textwidth]{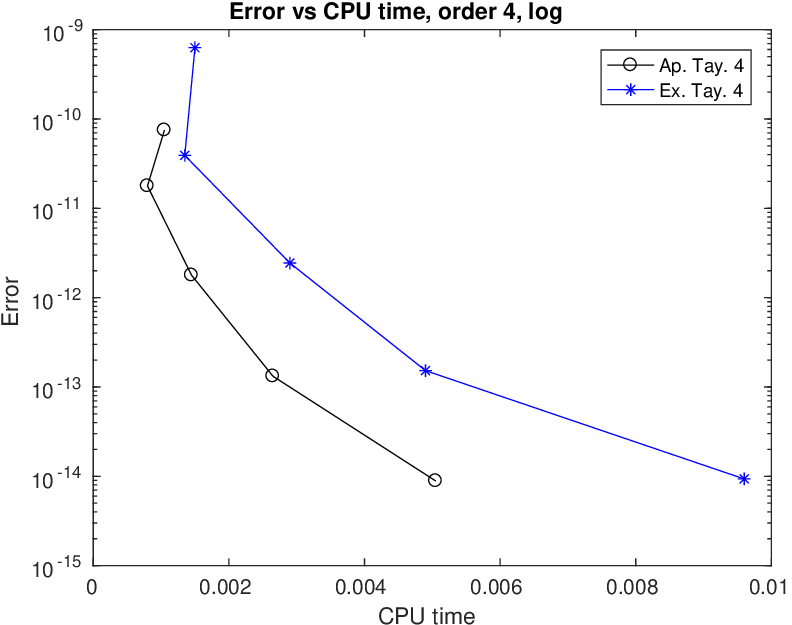}&
      \includegraphics[width=0.5\textwidth]{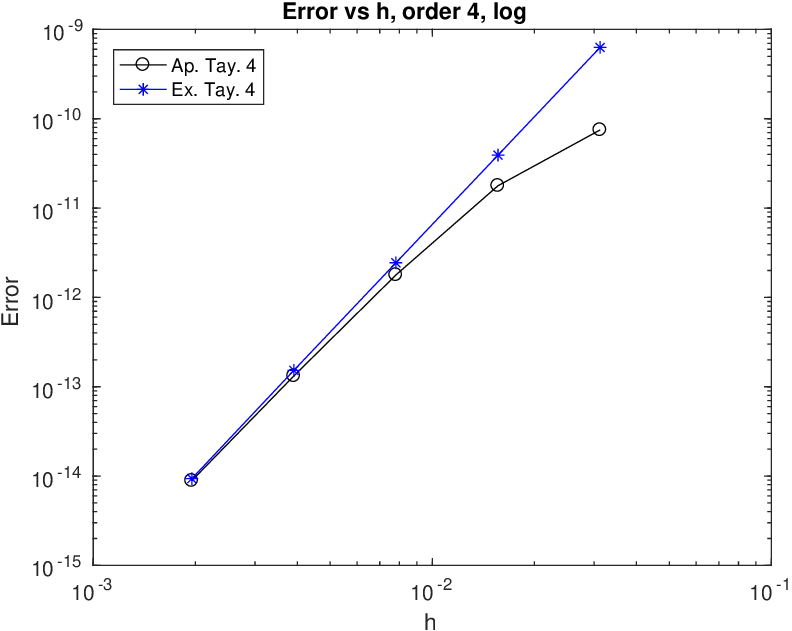}\\
      (c)&(d)\\
      \includegraphics[width=0.5\textwidth]{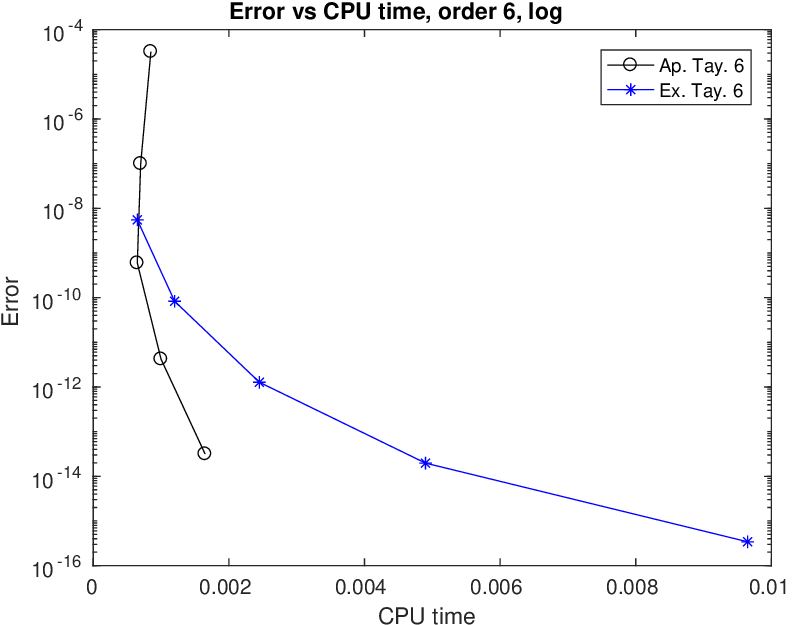}&
      \includegraphics[width=0.5\textwidth]{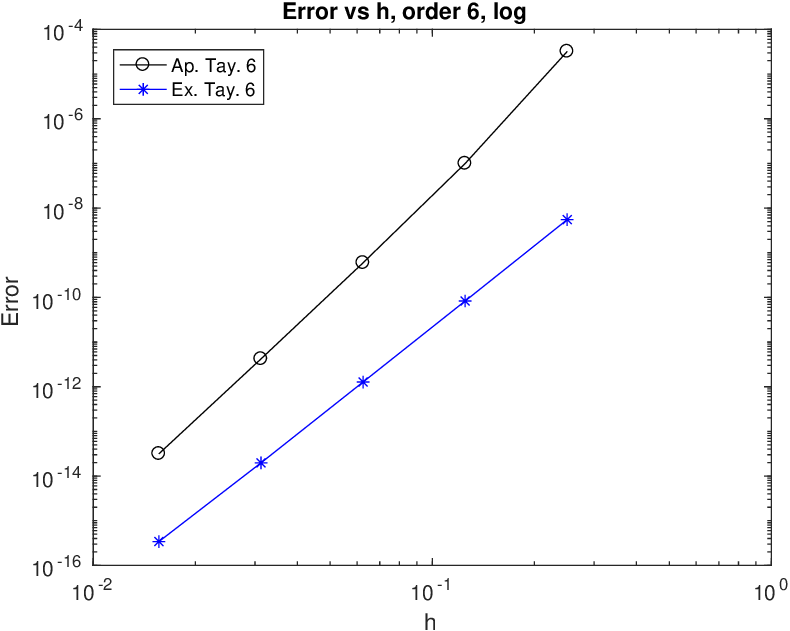}\\
      (e)&(f)	
    \end{tabular}
        \caption{Problem \eqref{eq:ex3}. Left: Error vs CPU time, Right: Error vs. $h$. (a)-(b) $R=2$; (c)-(d) $R=4$; (e)-(f) $R=6$.}
    \label{fig:log1}
  \end{center}
\end{figure}

From these experiments we observe that for both problems \eqref{eq:ex1} and \eqref{eq:ex2} the approximate method
outperforms the exact one for low orders, being the latter more
competitive as the order increases.  
This is probably due to the fact that the terms that appear in the Taylor expansion of $u$ involve 
many repeated trigonometric functions that are evaluated at the same point, thus taking advantage of 
simplifications in the expressions. In contrast, in the third example, corresponding to problem \eqref{eq:ex3}
 the approximate method obtains a sustained higher
 performance, as the Taylor terms in this case involve rational functions that contain a large number of 
 factors and do not simplify as easily as in the other examples. 
 We also observe in the right plots corresponding to the third example a discrepancy in 
 the absolute errors obtained by both methods that was barely observable in the previous examples.

\subsection{Systems of ODE}

We consider two problems modeled by systems of ODE. The first example models a pendulum suspended on an elastic 
spring or cord. Its position $(r_1(t), r_2(t))$ and 
velocity $(v_1(t), v_2(t))$ can be modeled with the following system:

\begin{equation}\label{eq:pendol}
\begin{array}{l}
r_1'=v_1,\\
r_2'=v_2,\\
v_1'=k_1(\frac{1}{\sqrt{r_1^2+r_2^2}}-1)r_1-k_2 v_1, \\
v_2'=k_1(\frac{1}{\sqrt{r_1^2+r_2^2}}-1)r_2-k_2 v_2 -g,
\end{array}
\end{equation}
where $k_1$ defines the elasticity of the cord, $k_2$ is related to
the friction of the pendulum and $g$ is the gravitational
acceleration. The initial conditions considered in this experiment are
$r_1(0)=0.7\,{\rm m}$,\ $r_2(0)=-0.8\,{\rm m}$,\ $v_1(0)=0.1\,{\rm m}/{\rm s}$, and $v_2(0)=-0.6\,{\rm m}/{\rm s}$,
for time $T=10\,$s and the parameters are taken as 
$m=1 \,{\rm Kg}$, $g=9.81 \,{\rm m}/{\rm s}^2$, $ k_1=100 \,{\rm N}/{\rm m}, k_2=1 \,{\rm Kg}/{\rm s}$. The results are
shown 
in Figure \ref{fig:pen1}.

\begin{figure}
  \begin{center}
    \begin{tabular}{cc}
      \includegraphics[width=0.5\textwidth]{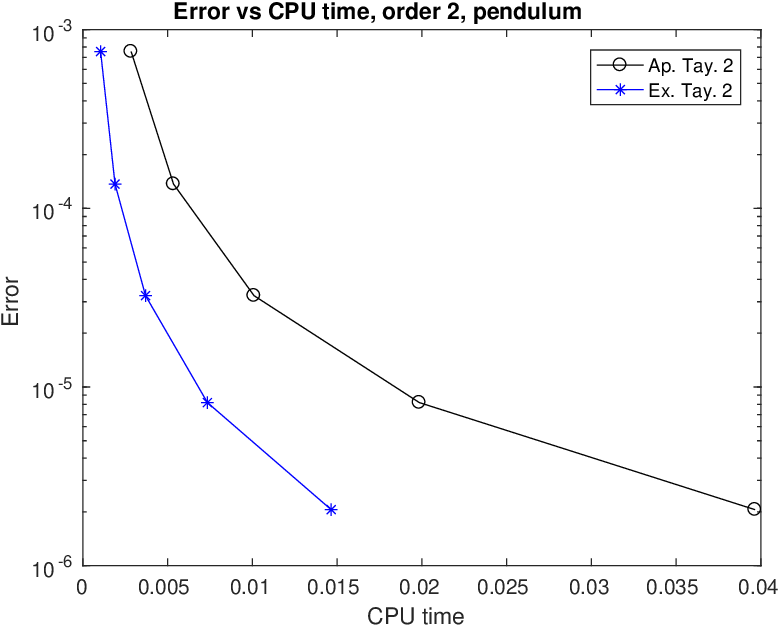}
      &\includegraphics[width=0.5\textwidth]{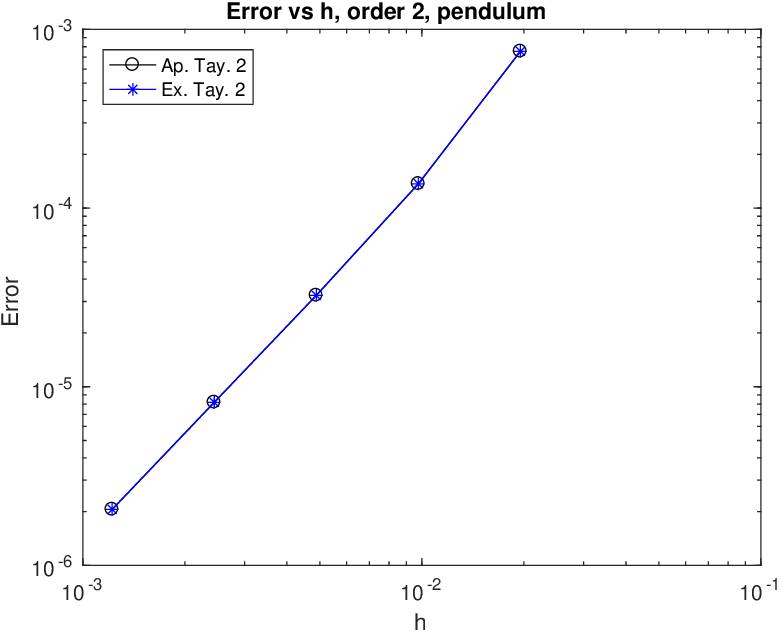}\\
      (a)&(b)\\
      \includegraphics[width=0.5\textwidth]{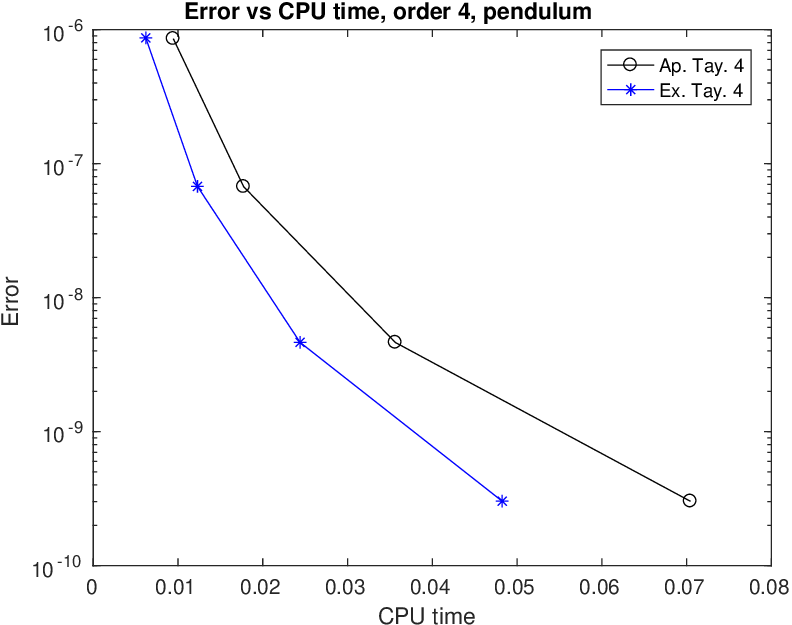}&
      \includegraphics[width=0.5\textwidth]{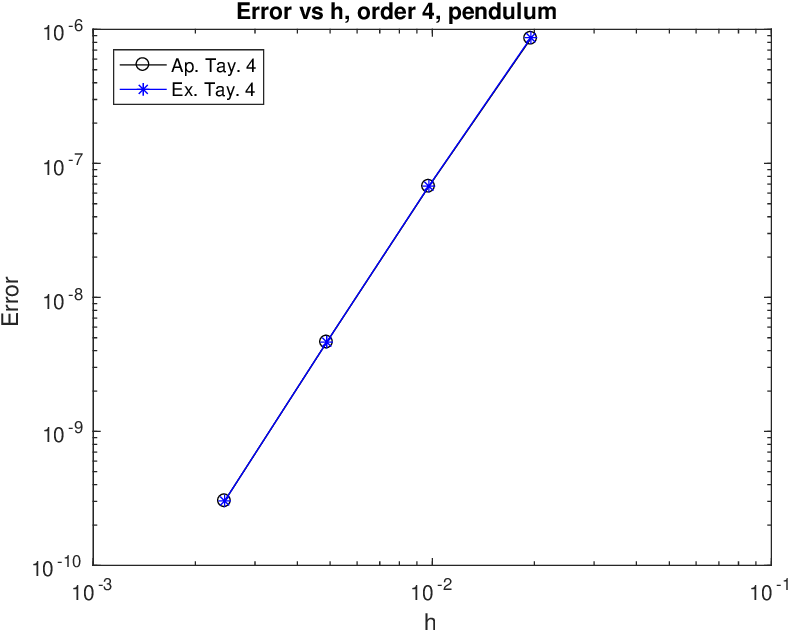}\\
      (c)&(d)\\
      \includegraphics[width=0.5\textwidth]{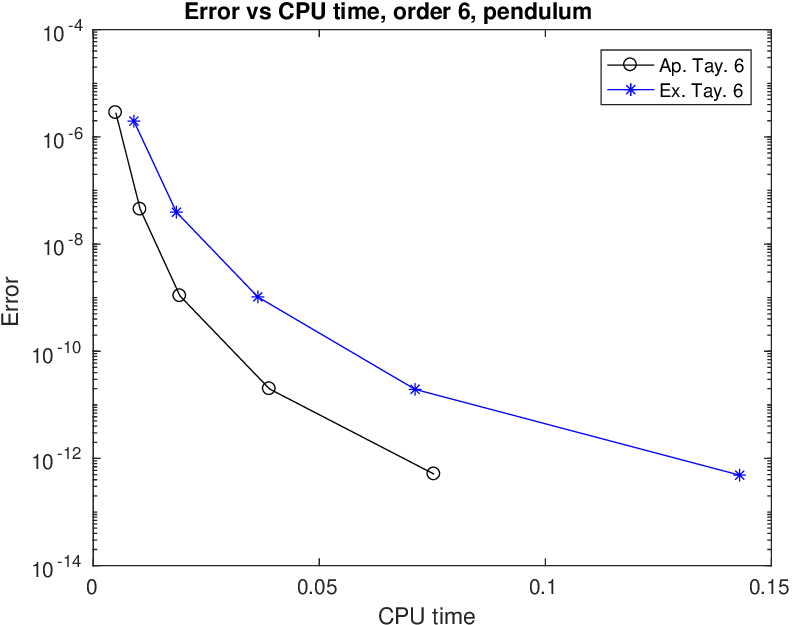}&
      \includegraphics[width=0.5\textwidth]{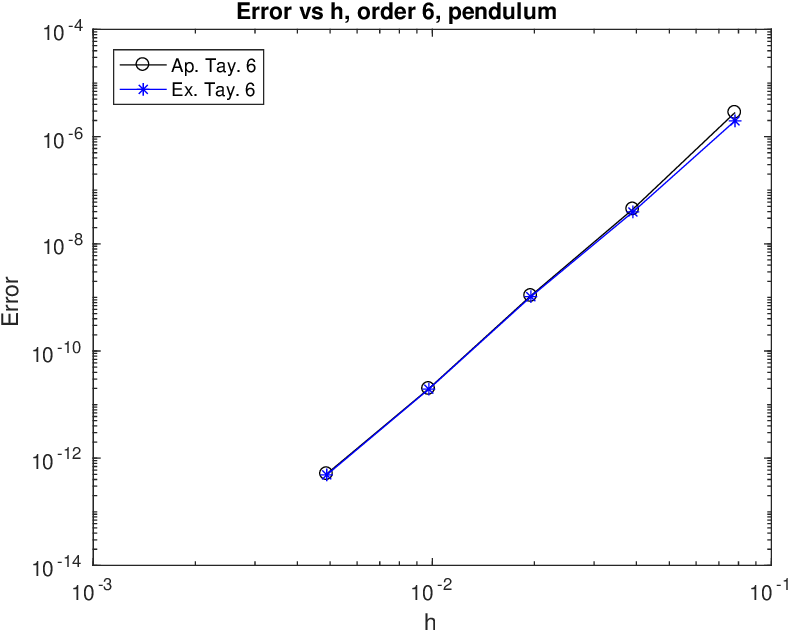}\\
      (e)&(f)\\
      \includegraphics[width=0.5\textwidth]{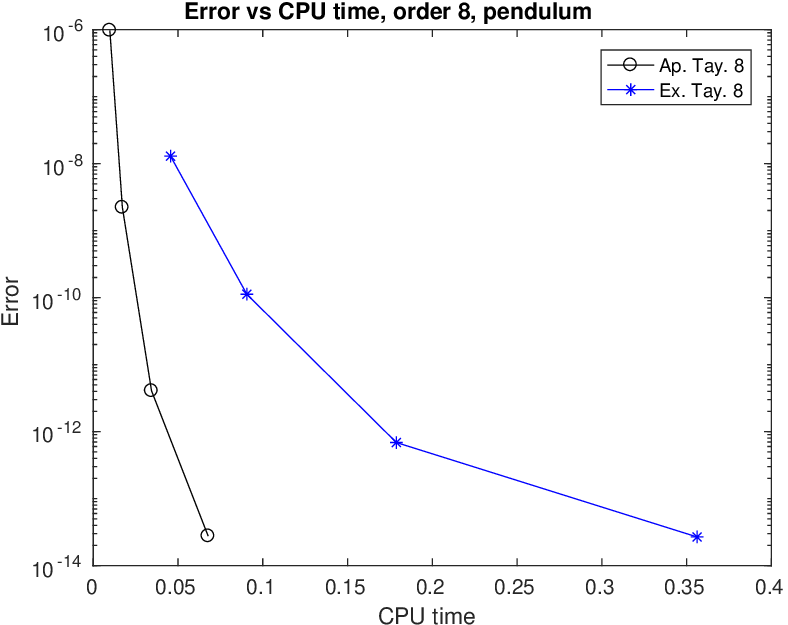}&
      \includegraphics[width=0.5\textwidth]{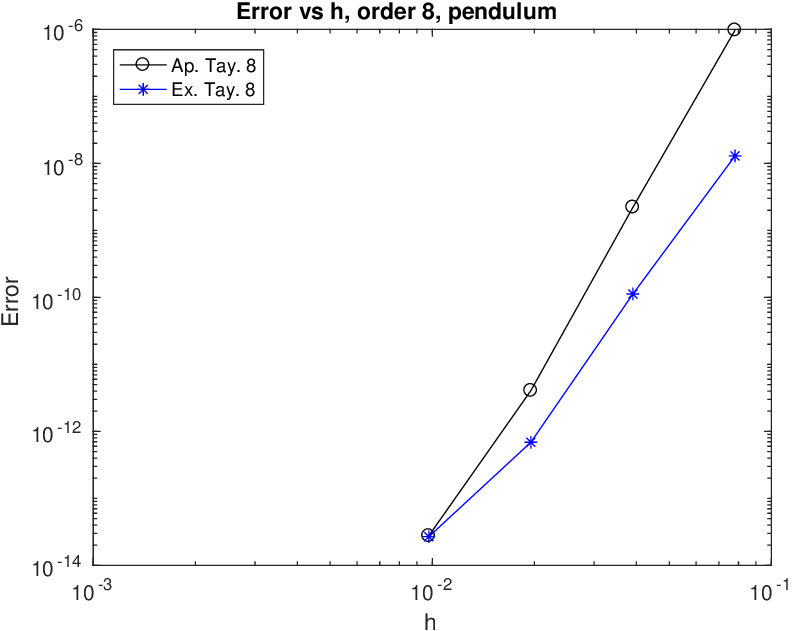}\\
      (g)&(h)	
    \end{tabular}
    \caption{Problem \eqref{eq:pendol}. Left: Error vs CPU time, Right: Error vs. $h$. (a)-(b) $R=2$; (c)-(d) $R=4$; (e)-(f) $R=6$;
    (g)-(h) $R=8$.}
    \label{fig:pen1}
  \end{center}
\end{figure}

Our next example corresponds to a biological model related to gene regulation, known as \emph{toggle switch}. 
The toggle switch models the interaction of two genes, indexed $L$ and $R$ that repress mutually:
\begin{equation}\label{eq:toggle}
\begin{aligned}
m_L' &= k_{L,1}\frac{\displaystyle K_R^{n_R}}{\displaystyle K_R^{n_R}+p_R^{n_R}}-d_{L,1}m_L,\\
p_L'&=k_{L,2}m_L-d_{L,2}p_L,\\
m_R' &= k_{R,1}\frac{\displaystyle K_L^{n_L}}{\displaystyle K_L^{n_L}+p_L^{n_L}}-d_{R,1}m_R,\\
p_R'&=k_{R,2}m_R-d_{R,2}p_R.
\end{aligned}
\end{equation}
 Quantities $m_{\{L,R\}}$
represent the concentration for gene $\{L,R\}$ Messenger Ribonucleic Acid (mRNA)  and $p_{\{L,R\}}$ the concentration of gene $\{L,R\}$ protein. The parameters
 are all set to 1 except $n_L = n_R = 2, k_{L,1} = k_{R,1}=10$, with initial conditions 
 $m_L(0)=0.5, p_L(0)=0.4, m_R(0)=0.5, p_R(0)=0.3$, for $T=10$.
Results for this problem can be seen in Figure \ref{fig:tog1}.

\begin{figure}
  \begin{center}
    \begin{tabular}{cc}
      \includegraphics[width=0.5\textwidth]{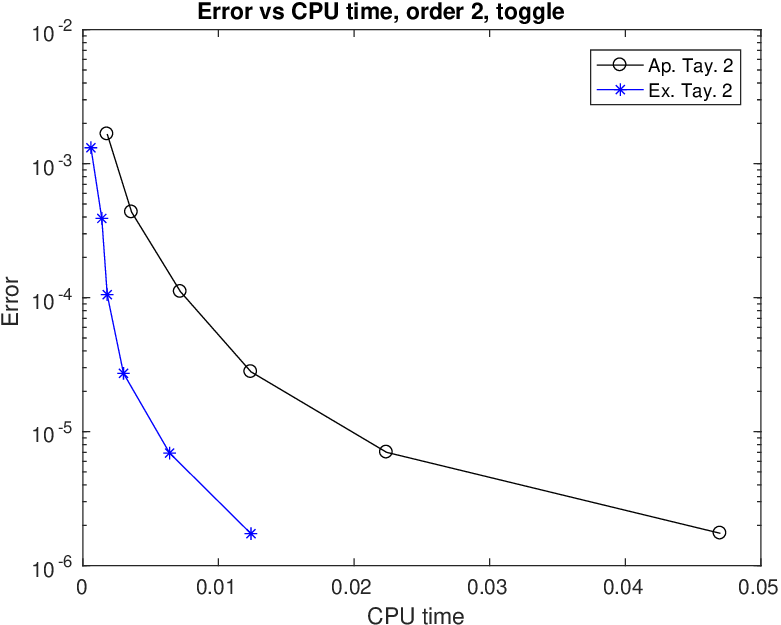}
      &\includegraphics[width=0.5\textwidth]{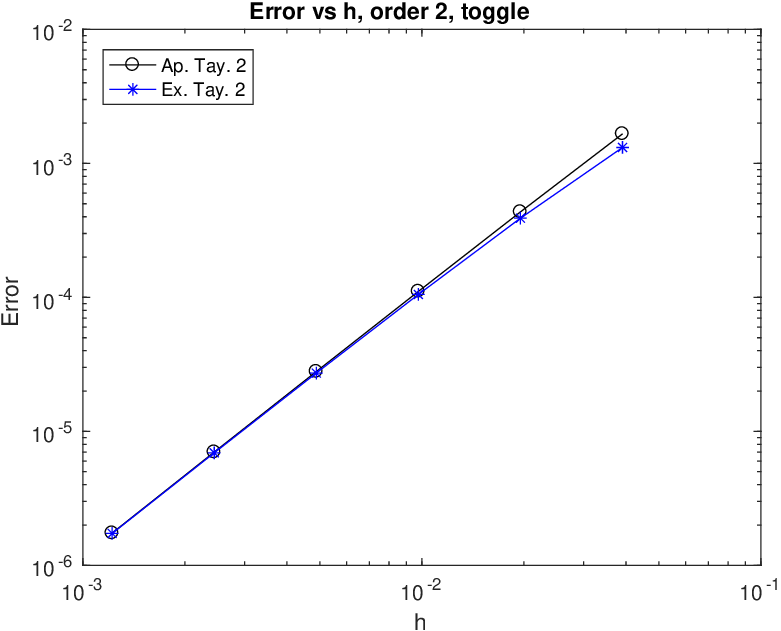}\\
      (a)&(b)\\
      \includegraphics[width=0.5\textwidth]{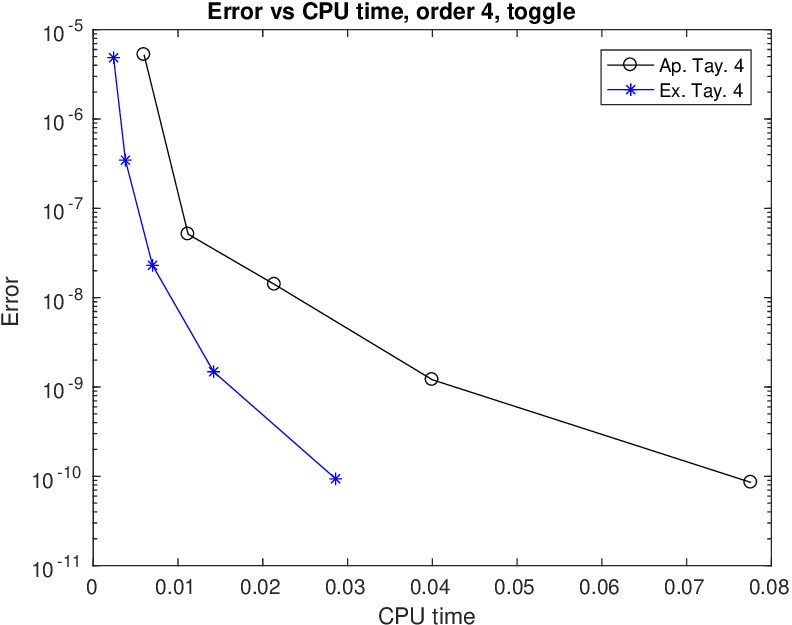}&
      \includegraphics[width=0.5\textwidth]{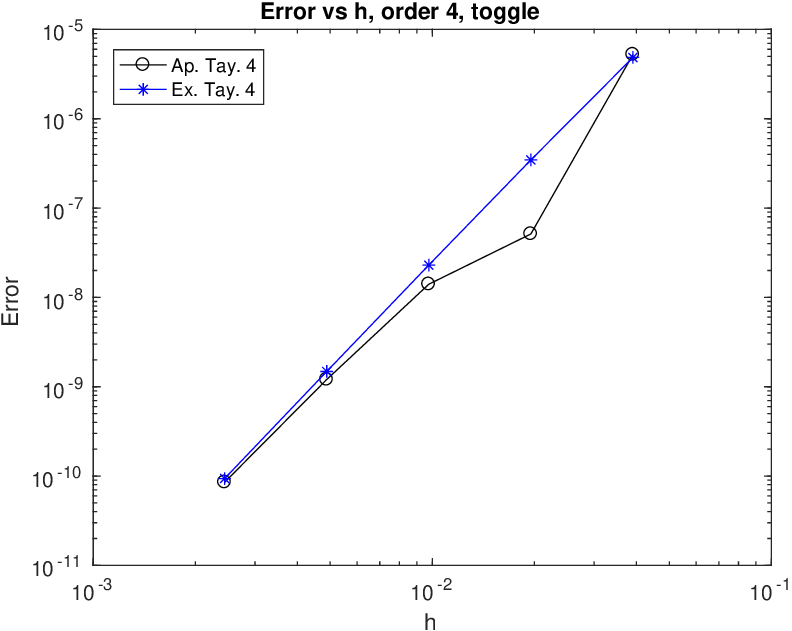}\\
      (c)&(d)\\
      \includegraphics[width=0.5\textwidth]{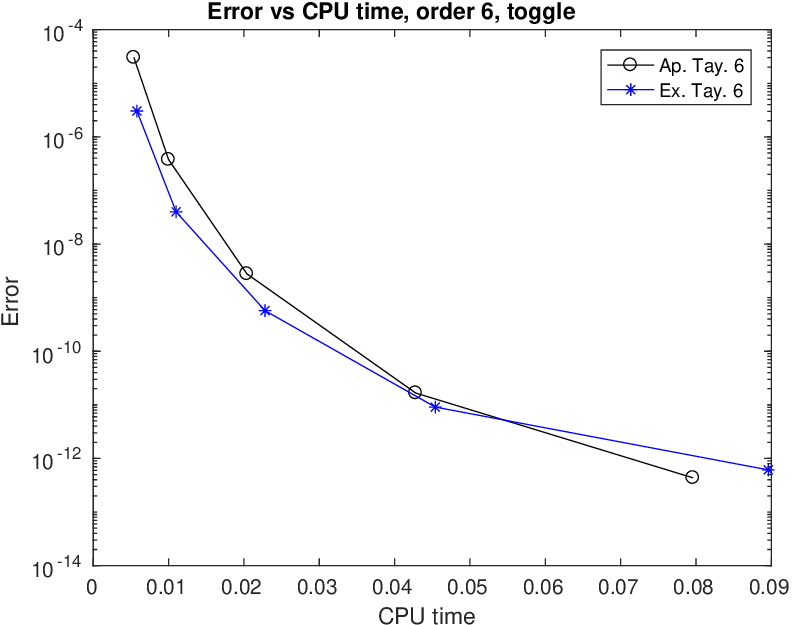}&
      \includegraphics[width=0.5\textwidth]{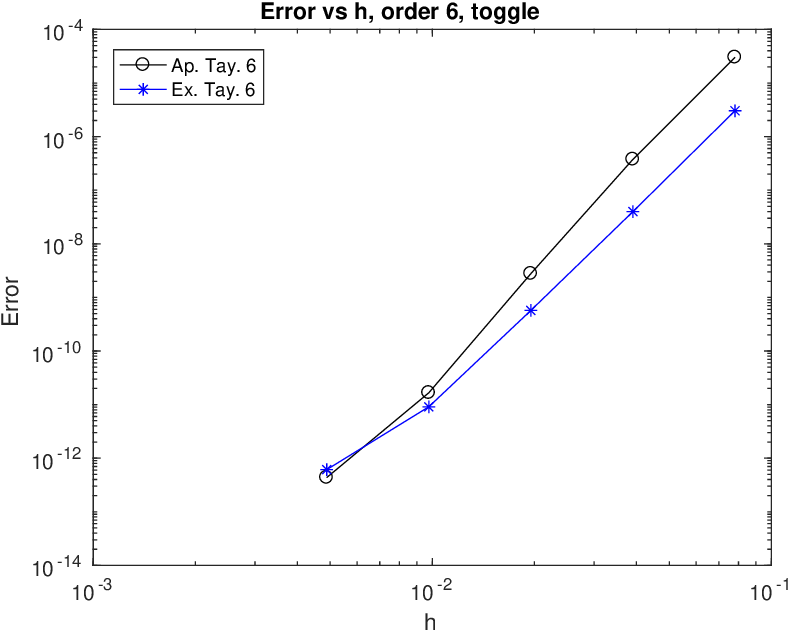}\\
      (e)&(f)\\
      \includegraphics[width=0.5\textwidth]{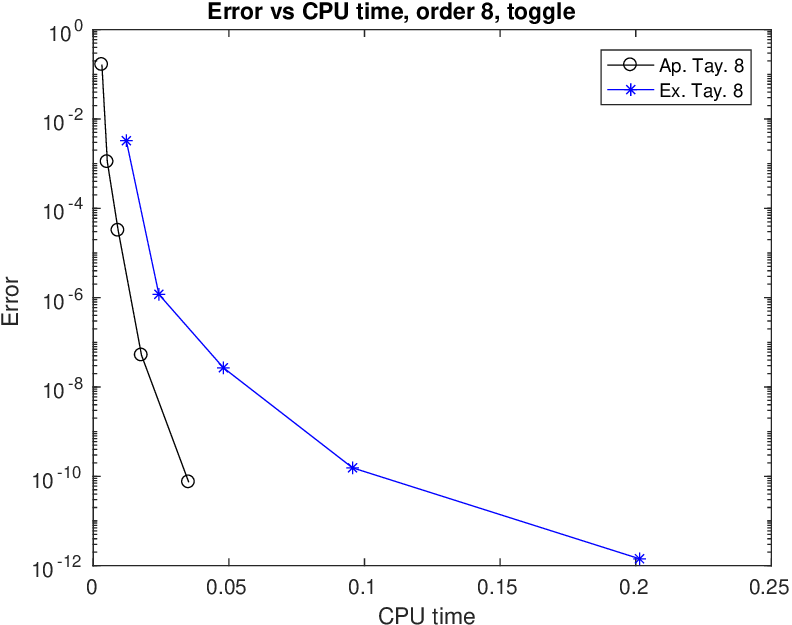}&
      \includegraphics[width=0.5\textwidth]{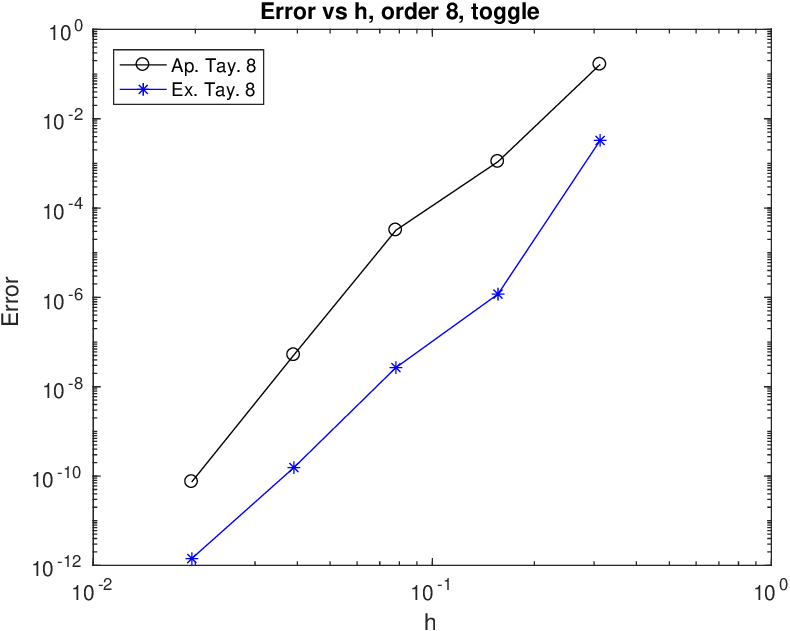}\\
      (g)&(h)	
    \end{tabular}
        \caption{Problem \eqref{eq:toggle}. Left: Error vs CPU time, Right: Error vs. $h$. (a)-(b) $R=2$; (c)-(d) $R=4$; (e)-(f) $R=6$;
    (g)-(h) $R=8$.}
    \label{fig:tog1}
  \end{center}
\end{figure}

  We consider also systems of ODE defined through homogeneous rational functions as used in Section \ref{sec:cc} to 
analyze the computational complexity of exact Taylor methods. Figure \ref{fig:rat1} shows the 
results corresponding to the solution of \eqref{eq:806} for the case $m=6$ for orders $2$, $4$, and $6$, for
time $T=10$ and initial conditions $u_i(0)=1,\ i=1,\dots,m$.

  \begin{figure}
  	\begin{center}
  		\begin{tabular}{cc}
  			\includegraphics[width=0.5\textwidth]{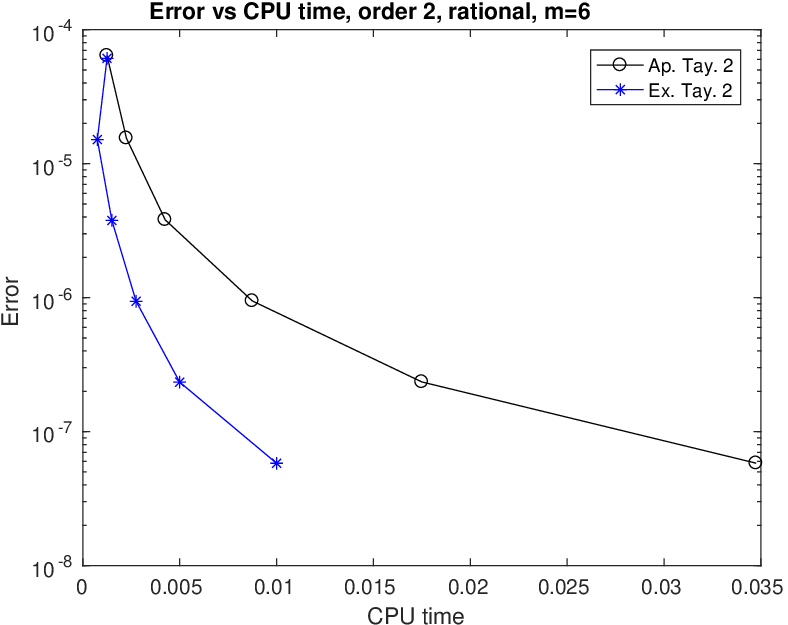}
  			&\includegraphics[width=0.5\textwidth]{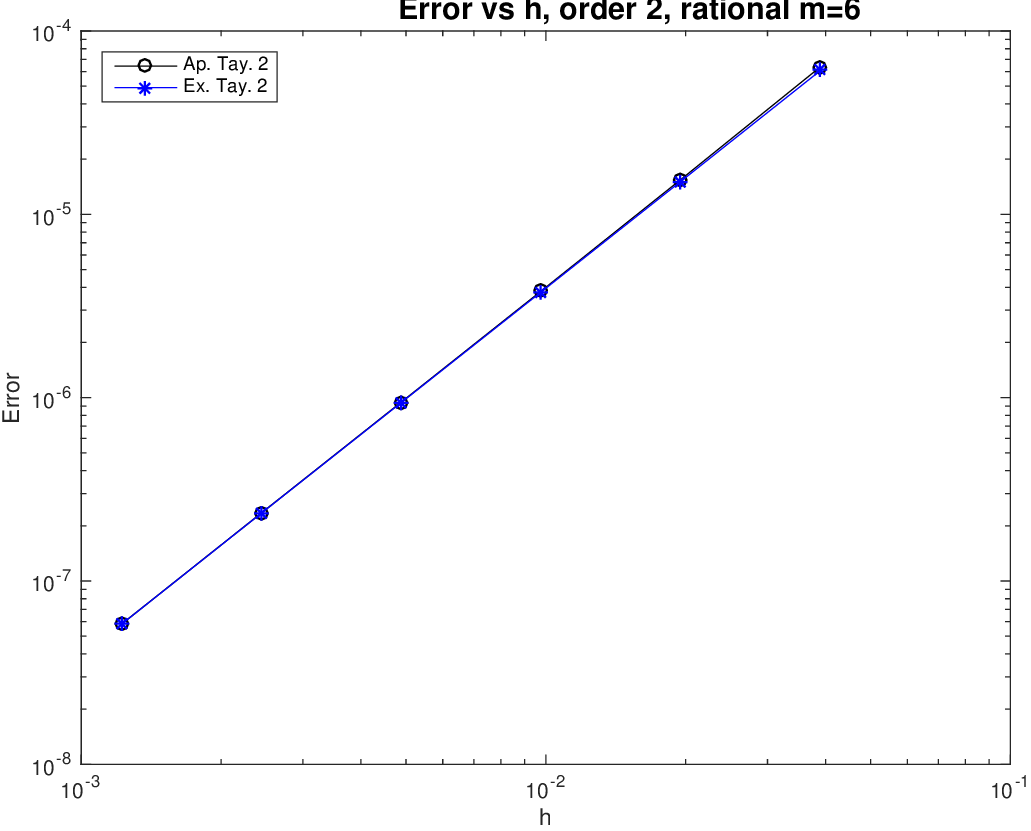}\\
  			(a)&(b)\\
  			\includegraphics[width=0.5\textwidth]{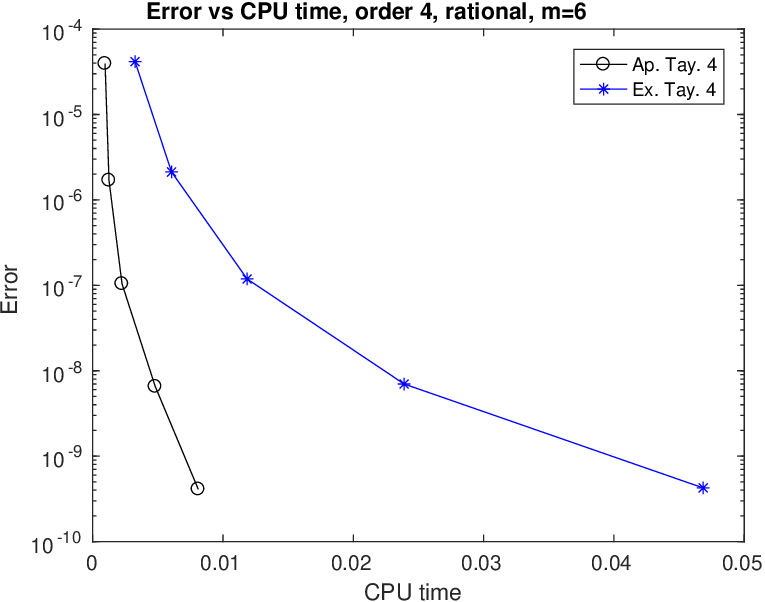}&
  			\includegraphics[width=0.5\textwidth]{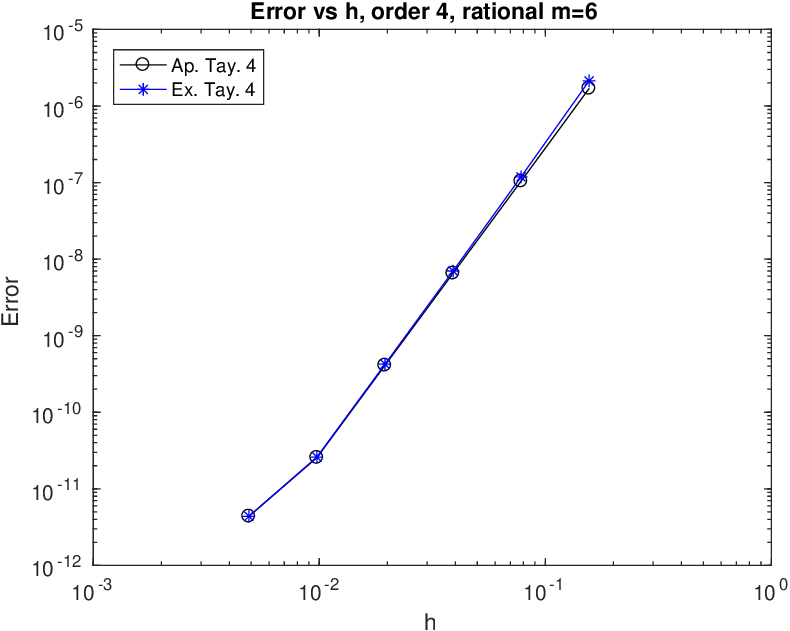}\\
  			(c)&(d)\\
  			\includegraphics[width=0.5\textwidth]{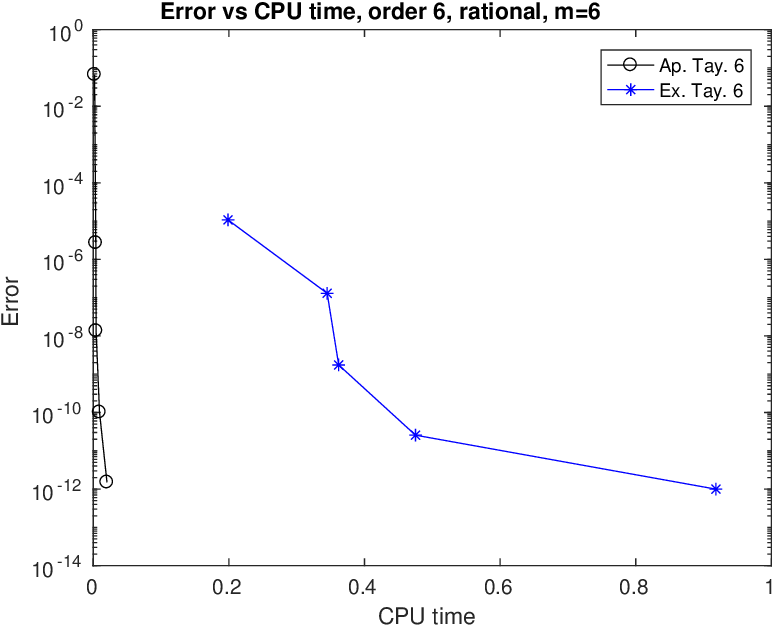}&
  			\includegraphics[width=0.5\textwidth]{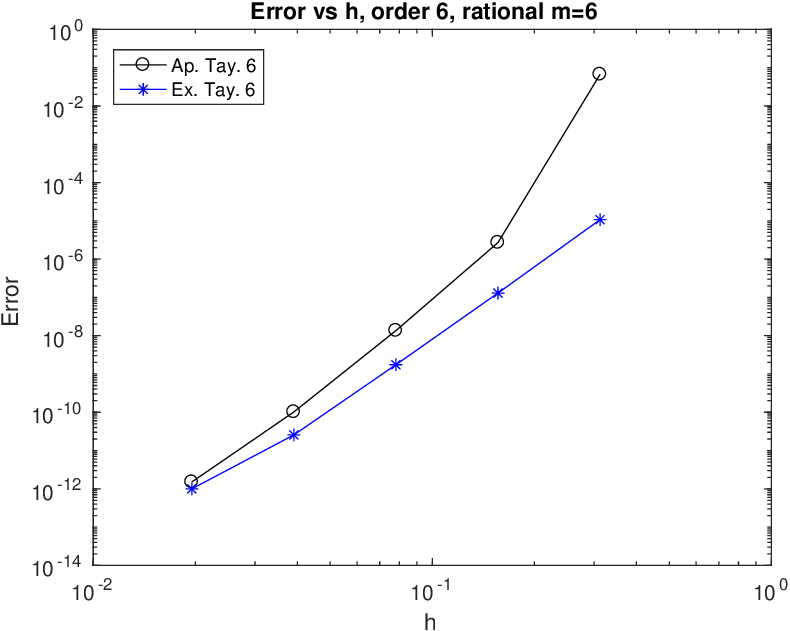}\\
  			(e)&(f)
  		\end{tabular}
  		\caption{Problem \eqref{eq:806}. Left: Error vs CPU time, Right: Error vs. $h$. (a)-(b) $R=2$; (c)-(d) $R=4$; (e)-(f) $R=6.$}
  		\label{fig:rat1}
  	\end{center}
  \end{figure}
  
\begin{table}
$$
\begin{array}{|r|r|r|r|r|}
  \hline
  &\multicolumn{4}{c|}{m}\\\hline
R&4 & 5 &6 &7  \\\hline
2 & 0.2288& 0.2564& 0.3026   & 0.4157  \\\hline
4 & 1.0531 &   2.5117& 5.6411 &  14.9425  \\\hline
6 & 8.4649   &28.6859& 58.5005&  110.6603  \\\hline
\end{array}
$$

\caption{Ratio of CPU time between the computational costs of
  exact and approximate $R$-th order Taylor methods for the rational system
  \eqref{eq:806}.
}
\label{tab:2}
\end{table}

In the case of systems we observe an increasing performance of the
approximate version  with respect to the exact version as the order increases.

In the last test  we consider the comparison of the numerical solution of \eqref{eq:806}
with the classical fourth-order Runge-Kutta method and the
fourth-order approximate Taylor method. We consider  the case $m=4$ for
time $T=10$ and initial conditions $u_i(0)=1,\ i=1,\dots,m$. Global errors are approximated from
a reference solution computed by the fourth order Runge-Kutta method
for $n=1000000$ time steps.

In  Figure \ref{fig:cmprk4rat1} we display the results of this
comparison. The expected conclusion is that the fourth-order Runge-Kutta method is
more efficient than the corresponding approximate Taylor method and
that both have fourth order global errors.

\begin{figure}
  \begin{center}
    \begin{tabular}{cc}
      \includegraphics[width=0.5\textwidth]{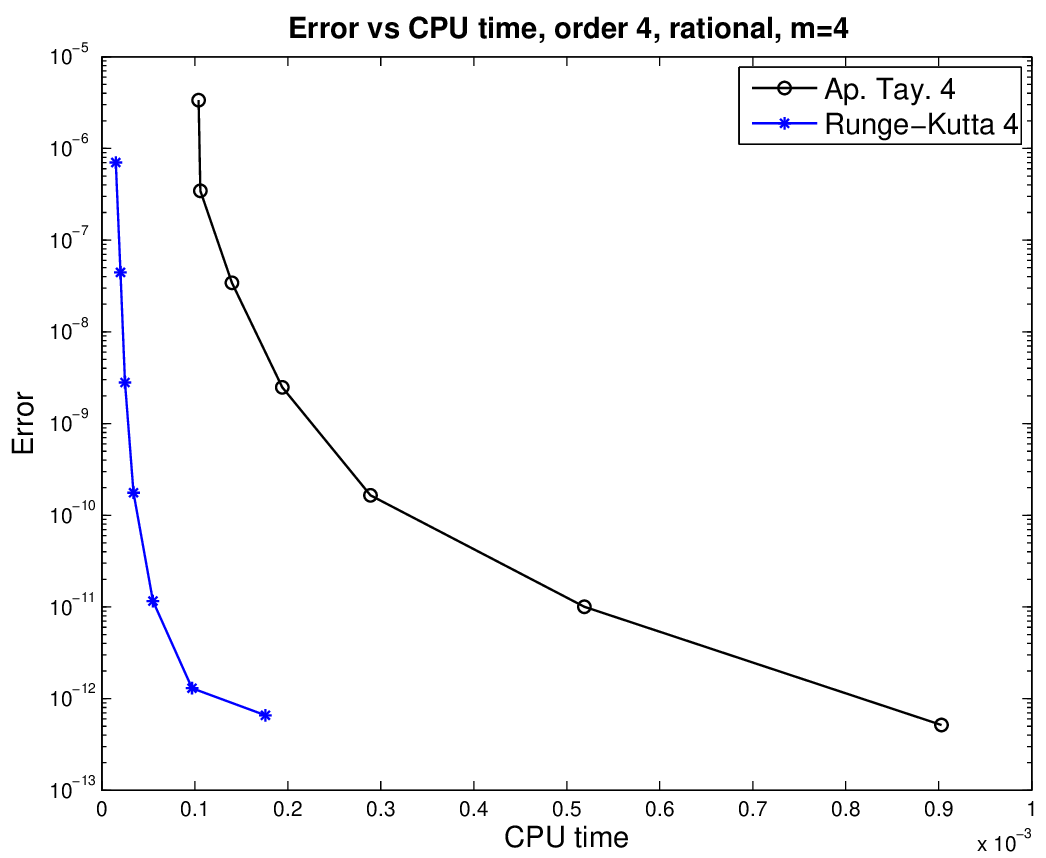}
      &\includegraphics[width=0.5\textwidth]{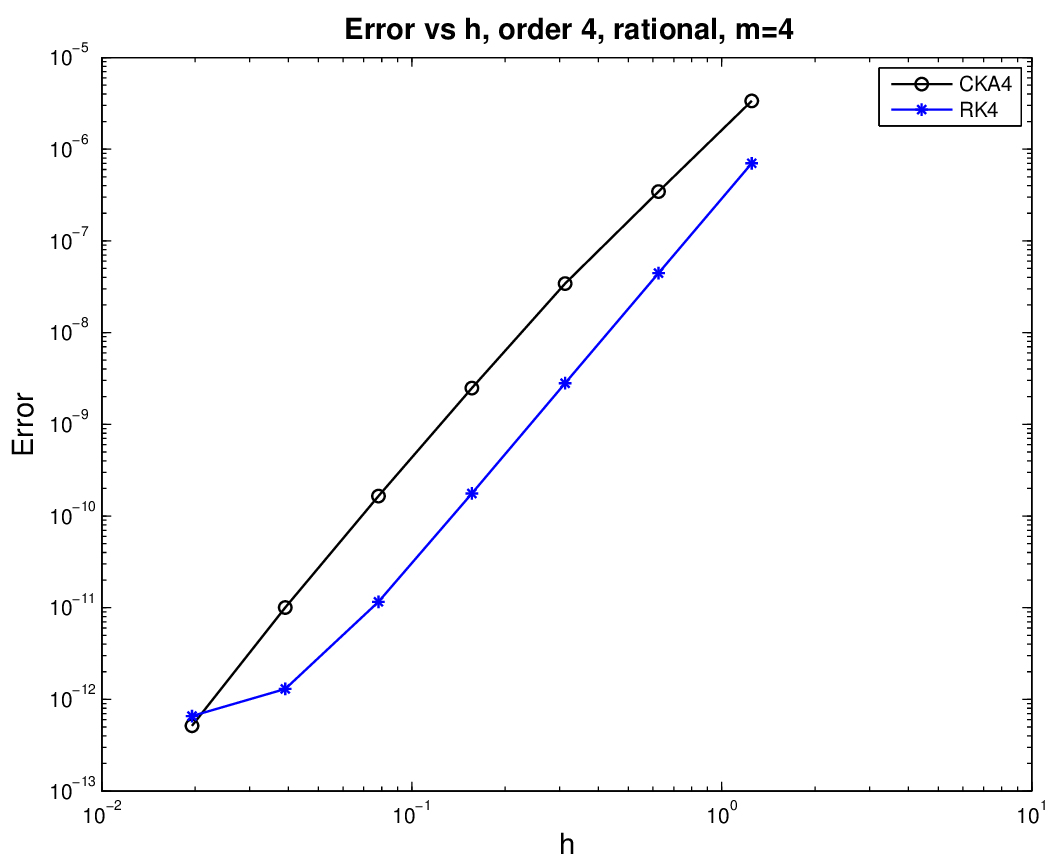}\\
      (a)&(b)
    \end{tabular}
  		\caption{Problem \eqref{eq:806}. Left: Error vs CPU
                  time, Right: Error vs. $h$. $R=4$, rational function $m=4$.}
  		\label{fig:cmprk4rat1}
  \end{center}
\end{figure}

\section{Conclusions}\label{cnc}
We have proposed a technique to obtain  ODE integrators  of arbitrary
order that can be regarded as  approximate Taylor methods. The
clear advantage of our proposal with respect to exact Taylor methods
is that only function evaluations are needed, so the cost of
implementation is very low. We have performed some
numerical experiments that show that  approximate Taylor methods
achieve the design order of accuracy,  both methods have roughly the same
accuracy and that approximate methods can be more computationally
efficient  in some cases, specially for moderate or large tightly
coupled ODE systems. As future research along the subject of approximate Taylor methods, we
consider the study of their relationship with Runge-Kutta methods and
their stability.

\section*{Acknowledgments}
Antonio Baeza, Pep Mulet and David Zor\'{\i}o are 
  supported by Spanish MINECO grant MTM 2014-54388-P. David Zor\'{\i}o is also
  supported by Fondecyt project 3170077. Giovanni Russo and Sebastiano Boscarino have been partially supported by ITN-ETN Horizon 2020 Project {\em ModCompShock, Modeling and Computation on Shocks and Interfaces}, Project Reference 642768, by the National Group for Scientific Computing INdAM-GNCS project 2017: \emph{Numerical methods for hyperbolic and kinetic equation and applications}.

\end{document}